\documentclass{amsart}
\usepackage{amsrefs}

\usepackage{graphicx}
\usepackage{epsfig}
\usepackage{mathptmx}
\usepackage{amssymb,amsfonts,mathrsfs,latexsym,textcomp}
\usepackage{amsmath}
\usepackage{dcolumn}

\usepackage{epstopdf}

\usepackage[english]{babel}
\usepackage{amssymb,amsmath,hyperref}
\usepackage{bbold}
\usepackage{stackrel}
\usepackage[foot]{amsaddr}
\usepackage{cleveref, enumerate}

\usepackage{transfer}

\newtheorem{theorem}{Theorem}
\newtheorem{cor}[theorem]{Corollary}
\newtheorem{definition}[theorem]{Definition}
\newtheorem{rem}[theorem]{Remark}
\newtheorem{nota}[theorem]{Notation}

\newtheorem{princ}[theorem]{Principle}

\newcommand\be{\begin{equation}}
\newcommand\ee{\end{equation}}

\newbox\gnBoxA
\newdimen\gnCornerHgt
\setbox\gnBoxA=\hbox{$\ulcorner$}
\global\gnCornerHgt=\ht\gnBoxA
\newdimen\gnArgHgt

\def\Godelnum #1{%
	\setbox\gnBoxA=\hbox{$#1$}%
	\gnArgHgt=\ht\gnBoxA%
	\ifnum \gnArgHgt<\gnCornerHgt
		\gnArgHgt=0pt%
	\else
		\advance \gnArgHgt by -\gnCornerHgt%
	\fi
	\raise\gnArgHgt\hbox{$\ulcorner$} \box\gnBoxA %
		\raise\gnArgHgt\hbox{$\urcorner$}}
\def\bnota{\begin{nota}\rm}
\def\enota{\end{nota}}
\def\brem{\begin{rem}\rm}
\def\erem{\end{rem}}

\def\ATR{\textup{\textsf{ATR}}}

\def\STP{\textup{\textsf{STP}}}

\def\ns{\textup{\textsf{ns}}}

\def\WKL{\textup{\textsf{WKL}}}

\def\T{\mathsf{T}}

\def\bye{\end{document}}

\def\P{\textup{\textsf{P}}}

\def\N{{\mathbb  N}}

\def\R{{\mathbb  R}}

\def\a{\mathbb{a}}

\def\GHU{\textup{\textsf{GHU}}}
\def\PCM{\textup{\textsf{PCM}}}
\def\NSC{\textup{\textsf{NPC}}}
\def\NPC{\textup{\textsf{NPC}}}
\def\MPC{\textup{\textsf{MPC}}}
\def\SCF{\textup{\textsf{SCF}}}
\def\GH{\textup{\textsf{GH}}}
\def\GHS{\textup{\textsf{GHS}}}

\def\MUC{\textup{\textsf{MUC}}}
\def\NUC{\textup{\textsf{NUC}}}

\def\EXT{\textup{\textsf{EXT}}}

\def\R{{\mathbb{R}}}
\def\({\textup{(}}
\def\){\textup{)}}

\def\st{\textup{st}}
\def\nst{\textup{nst}}
\def\asa{\leftrightarrow}

\def\di{\rightarrow}

\def\eps{\varepsilon}
\def\epsa{\varepsilon_{\alpha}}

\def\paai{\Pi_{1}^{0}\textup{-\textsf{TRANS}}}

\def\sigtoe{\Sigma_{2}^{0}\textup{\textsf{-TRANS}}}

\def\QFAC{\textup{\textsf{QF-AC}}}

\def\HER{\textup{\textsf{HER}}}
\def\ST{\textup{\textsf{ST}}}

\def\GH{\textup{\textsf{GH}}}
\def\OCA{\Omega\textup{\textsf{-CA}}}

\def\MU{\textup{\textsf{MU}}}

\def\HAC{\textup{\textsf{HAC}}}

\def\INT{\textup{\textsf{int}}}

\numberwithin{equation}{section}
\numberwithin{theorem}{section}

\usepackage{comment}

\numberwithin{equation}{section}

\usepackage{comment, enumerate}

\def\BI{\textup{\textsf{BI}}}
\def\EBI{\textup{\textsf{EBI}}}

\begin{document}

\title[The Gandy-Hyland functional and Nonstandard Analysis]{The Gandy-Hyland functional and a computational aspect of Nonstandard Analysis}
\author{Sam Sanders}
\address{Munich Center for Mathematical Philosophy, LMU Munich, Germany \& Department of Mathematics, Ghent University, Belgium}
\email{sasander@me.com}

%
%
%
%


\maketitle

%

\begin{abstract}
In this paper, we highlight a new \emph{computational} aspect of Nonstandard Analysis relating to higher-order computability theory. 
In particular, we prove that the \emph{Gandy-Hyland functional} equals a primitive recursive functional \emph{involving nonstandard numbers} inside Nelson's \emph{internal set theory}.     
From this classical and ineffective proof in Nonstandard Analysis, a term from G\"odel's system \textsf{T} can be extracted which computes the Gandy-Hyland functional in terms of a \emph{modulus-of-continuity functional} and a special case of the \emph{fan functional}.  
We obtain several similar relative computability results \emph{not involving Nonstandard Analysis} from their associated nonstandard theorems.  
By way of reversal, we show that certain relative computability results, called \emph{Herbrandisations}, also imply the nonstandard theorem from whence they were obtained.  
Thus, we establish a direct two-way connection between the field \emph{Computability} (in particular theoretical computer science) and the field \emph{Nonstandard Analysis}.      
\end{abstract}

\maketitle

\section{Introduction}\label{intro}
Our aim is to highlight a new \emph{computational} aspect of Nonstandard Analysis relating to (higher-order) computability theory. 
We study the \emph{Gandy-Hyland functional}, which was introduced in \cite{gandymahat} as an example of a higher-type functional 
not computable, in the sense of Kleene's S1-S9 (See \cite{noortje}*{1.10} or \cite{longmann}*{5.1.1}), in the \emph{fan functional} over the total continuous functionals (See \cite{noortje}*{4.61} or \cite{longmann}*{8.3.3}).   
The {Gandy-Hyland functional} $\Gamma$ is defined as follows:  
\begin{equation}\label{GH}\tag{\textsf{GH}}
(\exists \Gamma^{3})(\forall Y^{2}\in C, s^{0})\big[\Gamma(Y^{2},s^{0})= Y\big(s*0* (\lambda n^{0})\Gamma(Y, s*(n+1))\big)\big], 
\end{equation}
where `$Y^{2}\in C$' is the definition of continuity on Baire space as in \eqref{CB2};  All notations are introduced in Section \ref{TITI}.
\be\label{CB2}
(\forall f^{1})(\exists N^{0})(\forall g^{1})(\overline{f}N=_{0}\overline{g}N\di Y(f)=_{0}Y(g)).  
\ee
The functional $\Gamma$ from \eqref{GH} apparently exhibits non-well-founded self-reference:  Indeed, in order to compute $\Gamma$ at $s^{0}$, one needs the values of $\Gamma$ at all child nodes of $s^{0}$, as is clear from the right-hand side of \eqref{GH}.  
In turn, to compute the value of $\Gamma$ at the child nodes of $s$, one needs the value of $\Gamma$ at all grand-child nodes of $s$, and so on.  Hence, repeatedly applying the definition of $\Gamma$  seems to result in a non-terminating recursion.       
By contrast, \emph{primitive recursion} is well-founded as it reduces the case for $n+1$ to the case for $n$, and the case for $n=0$ is given.  

\medskip

In Section \ref{deballen}, we show that in Nelson's \emph{internal set theory} (See Section~\ref{P}), the \emph{primitive recursive}\footnote{The functional $G$ is primitive recursive \emph{in the sense of G\"odel's system} $\T$;  See Section \ref{TITI}.} functional 
\be\label{small}
G(Y,s,M)=
\begin{cases}
Y(s*00\dots) & |s|\geq M \\
Y(s*0* (\lambda n^{0})G(Y, s*(n+1),M)) & \textup{otherwise}   
\end{cases}
\ee
equals the $\Gamma$-functional from \eqref{GH} for standard input and any nonstandard number $M^{0}$.  
Note that one need only apply the definition of $G$ at most $M$ times to terminate in the first case of \eqref{small}.  In other words, the extra case `$|s|\geq M$' 
provides a nonstandard stopping condition which `unwinds' the non-terminating recursion in $\Gamma$ to the terminating one in $G$.  Or: one can trade in self-reference for nonstandard numbers.  
Thus, we shall refer to $G$ as the \emph{canonical approximation} of $\Gamma$.  

\medskip

We work in $\P$, a fragment of Nelson's \emph{internal set theory} based on G\"odel's $\textsf{T}$, both introduced in Section \ref{T}.  
The proof in Section \ref{drei} that $G(\cdot, M)$ and $\Gamma(\cdot)$ are equal for standard inputs and nonstandard $M^{0}$, 
takes place in \textsf{P} augmented with a nonstandard continuity axiom \ref{druk2} and a nonstandard bar induction axiom \ref{STP}.   
This is a natural setting for $\Gamma$, as it is modified bar recursion in disguise (\cite{bergolijf}*{\S4}).  

\medskip

From the aforementioned proof in $\P$ regarding $\Gamma$ and $G$, we show how to extract a term $t$ from G\"odel's $\textsf{T}$ and a proof in higher-order Peano arithmetic that $t$ 
computes the Gandy-Hyland functional in terms of a special case of the fan functional and 
a modulus-of-continuity functional.  
Conceptually, it is important to note that this final proof, as well as the term $t$, \emph{does not involve Nonstandard Analysis}, and that the extraction of the term $t$ from the proof proceeds via an algorithm.    
In Sections~\ref{laiba} to~\ref{quatre}, we obtain further nonstandard results from which we extract related relative computability results.  
In particular, the relative computability results in Section \ref{laiba} and \ref{ballen3} are `pointwise' versions of previous results in this paper, and these results witness the robustness of our approach.  Furthermore, in Section \ref{ballen3} and \ref{quatre}, we prove some \emph{Reverse Mathematics}$^{\ref{lotp}}$ results (both nonstandard and effective).  In particular, we work with \emph{associates} of continuous functionals in Section \ref{quatre}, leading to particularly elegant results.      

\medskip

Finally, it is a natural `Reverse Mathematics\footnote{For Friedman's foundational program \emph{Reverse Mathematics}, we refer the reader to Simpson's monograph \cite{simpson2}, which provides an excellent overview.\label{lotp}} style' question if from a relative computability result, obtained via Nonstandard Analysis, the `original' nonstandard theorem can be re-obtained.  
In answer to this question, we show in Section~\ref{relletje} that (a proof of) the original nonstandard theorem 
(that the Gandy-Hyland functional $\Gamma(\cdot)$ equals $G(\cdot, M)$ for all standard inputs and nonstandard $M^{0}$) follows from (a proof of) a certain natural relative computability result, called the \emph{Herbrandisation} 
of the original nonstandard theorem.  

\medskip

In conclusion, while these relative computability results are not necessarily deep or surprising in and of themselves, 
\emph{the methodology by which we arrive at them} constitutes the real surprise of this paper, namely \emph{a new computational aspect of Nonstandard Analysis}:  
From a classical-logic proof in which no attention to computability is given at all, and in which Nonstandard Analysis is freely used, we obtain a relative computability result in a straightforward way.   
With some attention to detail, a natural relative computability result, called the Herbrandisation, allows us to re-obtain the original nonstandard theorem. 
In this way, we establish a direct two-way connection between the field Computability (in particular theoretical computer science) and Nonstandard Analysis.      
As such, our results differ significantly from known applications of Nonstandard Analysis in Computability, as discussed in Section \ref{XYZ}.  

\section{About and around internal set theory}\label{T}
In this section, we introduce the system $\P$ in which we will work.  
In a nutshell, $\P$ is a conservative extension of G\"odel's system $\textsf{T}$ with certain axioms from Nelson's \emph{Internal Set Theory} $\IST$ (\cite{wownelly}) based on the approach from \cite{brie,bennosam}.      

\subsection{Nelson's internal set theory $\IST$}\label{P}
In this section, we discuss Nelson's \emph{internal set theory} $\IST$, first introduced in \cite{wownelly}.  
In Nelson's \emph{syntactic} approach to Nonstandard Analysis, as opposed to Robinson's semantic one (\cite{robinson1}), a new unary predicate `st($x$)', read as `$x$ is standard', is added to the language of \textsf{ZFC}, the usual foundation of mathematics.  In other words, $L_{\IST}=L_{\ZFC}\cup\{\st\}$ is the language of $\IST$ if $L_{\ZFC}$ is the language of $\ZFC$.  The notations $(\forall^{\st}x)(\dots)$ and $(\exists^{\st}y)(\dots)$ are short for $(\forall x)(\st(x)\di \dots)$ and $(\exists y)(\st(y)\wedge \dots)$.  
A formula of $\IST$ is called \emph{internal} if it does not involve `st', and \emph{external} otherwise.   
The external axioms \emph{Idealisation}, \emph{Standard Part}, and \emph{Transfer} govern `st', and are defined\footnote{The superscript `fin' in \textsf{(I)} means that $x$ is `finite', i.e.\ the set is in one-to-one correspondence with some subset $\{0, 1, \dots, N\}$ of the natural numbers.  A finite set $x$ is standard in $\IST$ if (and only if) the associated number $N$ is, in light of \cite{wownelly}*{Theorem 1.1}.} as:
\begin{definition}[Axioms of $\IST$]~
\begin{enumerate}
\item[\textsf{\textup{(I)}}] $(\forall^{\st~\textup{fin}}x)(\exists y)(\forall z\in x)\varphi(z,y)\di (\exists y)(\forall^{\st}x)\varphi(x,y)$, for internal $\varphi$. 
\item[\textsf{\textup{(S)}}] $(\forall x^{\st})(\exists^{\st}y)(\forall^{\st}z)\big((z\in x\wedge \varphi(z))\asa z\in y\big)$, for any formula $\varphi$.
\item[\textsf{\textup{(T)}}] $(\forall^{\st}t)\big[(\forall^{\st}x)\varphi(x, t)\di (\forall x)\varphi(x, t)\big]$, where $\varphi(x, t)$ is internal and only has free variables $x,t$.  
\end{enumerate}
\end{definition}
The system \textsf{IST} consist of the aforementioned three external axioms, plus the internal system \textsf{ZFC}, i.e. the latter does not involve `st'.
Now, $\IST$ is a conservative extension of \textsf{ZFC} for the internal language $L_{\ZFC}$, as proved in \cite{wownelly}.    
Of course, the step from $\textsf{ZFC}$ to $\IST$ can be done for \emph{fragments} of $\IST$ too; we shall make use the system $\P$, a fragment of $\IST$ based on G\"odel's system $\textsf{T}$, introduced in Section \ref{PIPI}.  
Before that, we briefly introduce system $\textsf{T}$ in Section \ref{TITI}.  The system $\P$ was first introduced in \cite{brie} and is exceptional in that it has a `term extraction procedure' \emph{with a very wide scope} (See Corollary \ref{consresultcor}).  We discuss this in more detail in Remark~\ref{firliborn}.        

\subsection{G\"odel's system $\textsf{T}$}\label{TITI}
In this section, we briefly introduce G\"odel's system $\textsf{T}$ and the associated systems $\textsf{E-PA}^{\omega}$ and $\textsf{E-PA}^{\omega*}$.  

\medskip

In his \emph{Dialectica} paper (\cite{godel3}), G\"odel defines an interpretation of intuitionistic arithmetic into a quantifier-free calculus of functionals.  This calculus is now known as `G\"odel's system \textsf{T}', and is essentially just primitive recursive arithmetic (\cite{buss}*{\S1.2.10}) with the schema of recursion expanded to \emph{all finite types}.  
Firstly, the set of finite types $\pmb{T}$ is:
\begin{center}
(i) $0\in \pmb{T}$   and   (ii)  If $\sigma, \tau\in \pmb{T}$ then $( \sigma \di \tau) \in \pmb{T}$,
\end{center}
where $0$ is the type of natural numbers, and $\sigma\di \tau$ is the type of mappings from objects of type $\sigma$ to objects of type $\tau$.  G\"odel's system $\textsf{T}$ includes `recursor' constants $\pmb{R}^{\rho}$ for every finite type $\rho\in \pmb{T}$, defining primitive recursion:  
\be\label{PR}\tag{\textsf{\textup{PR}}}
\pmb{R}^{\rho}(f,g, 0):=f   \textup{ and } \pmb{R}^{\rho}(f, g, n+1):=g(n, \pmb{R}^{\rho}(f, g, n)),
\ee
for $f^{\rho}$ and $g^{0\di( \rho\di \rho)}$.  
The system $\textsf{E-PA}^{\omega}$ is a combination of \emph{Peano arithmetic} and system $\T$, and the full axiom of extensionality \eqref{EXT}.  
The detailed definition of $\textsf{E-PA}^{\omega}$ may be found in \cite{kohlenbach3}*{\S3.3};  We do introduce the notion of equality and extensionality in $\textsf{E-PA}^{\omega}$, as these notions are needed below.
\begin{definition}[Equality]\label{FAK}
The system $\textsf{E-PA}^{\omega}$ includes equality between natural numbers `$=_{0}$' as a primitive.  Equality `$=_{\tau}$' for type $\tau$-objects $x,y$ is then:
\be\label{aparth}
[x=_{\tau}y] \equiv (\forall z_{1}^{\tau_{1}}\dots z_{k}^{\tau_{k}})[xz_{1}\dots z_{k}=_{0}yz_{1}\dots z_{k}]
\ee
if the type $\tau$ is composed as $\tau\equiv(\tau_{1}\di \dots\di \tau_{k}\di 0)$.  
The usual inequality predicate `$\leq_{0}$' between numbers has an obvious definition, and the predicate `$\leq_{\tau}$' is just `$=_{\tau}$' with `$=_{0}$' replaced by `$\leq_{0}$' in \eqref{aparth}.    
The \emph{axiom of extensionality} is the statement that for all $\rho, \tau\in \pmb{T}$:
\be\label{EXT}\tag{\textsf{E}}  
(\forall  x^{\rho},y^{\rho}, \varphi^{\rho\di \tau}) \big[x=_{\rho} y \di \varphi(x)=_{\tau}\varphi(y)   \big], 
\ee 
\end{definition}
Next, we introduce $\textsf{E-PA}^{\omega*}$, a definitional extension of $\textsf{E-PA}^{\omega}$ with a type for finite sequences; The set $\pmb{T}^{*}$ is:
\begin{center}
(i) $0\in \pmb{T}^{*}$,   (ii)  If $\sigma, \tau\in \pmb{T}^{*}$ then $ (\sigma \di \tau) \in \pmb{T}^{*}$, and (iii) If $\sigma \in \pmb{T}^{*}$, then $\sigma^{*}\in \pmb{T}^{*}$,
\end{center}
where $\sigma^{*}$ is the type of finite sequences of objects of type $\sigma$.  The system $\textsf{E-PA}^{\omega*}$ includes \eqref{PR} for all $\rho\in \pmb{T}^{*}$, as well as dedicated `list recursors' to handle finite sequences for any $\rho^{*}\in \pmb{T}^{*}$.        
A detailed definition of $\textsf{E-PA}^{\omega*}$ may be found in \cite{brie}*{\S2.1}.  
We now introduce some notations specific to $\textsf{E-PA}^{\omega*}$, as also used in \cite{brie}.
\begin{nota}[Finite sequences]\label{skim}\rm
The system $\textsf{E-PA}^{\omega*}$ has a dedicated type for `finite sequences of objects of type $\rho$', namely $\rho^{*}$.  Since the usual coding of pairs of numbers goes through in $\textsf{E-PA}^{\omega*}$, we shall not always distinguish between $0$ and $0^{*}$;  See e.g.\ the definition $(\GH)$ of the Gandy-Hyland functional in Section \ref{intro}.  
Similarly, we do not always distinguish between `$s^{\rho}$' and `$\langle s^{\rho}\rangle$', where the former is `the object $s$ of type $\rho$', and the latter is `the sequence of type $\rho^{*}$ with only element $s^{\rho}$'.  The empty sequence for the type $\rho^{*}$ is denoted by `$\langle \rangle_{\rho}$', usually with the typing omitted.  Furthermore, we denote by `$|s|=n$' the length of the finite sequence $s^{\rho^{*}}=\langle s_{0}^{\rho},s_{1}^{\rho},\dots,s_{n-1}^{\rho}\rangle$, where $|\langle\rangle|=0$, i.e.\ the empty sequence has length zero.
For sequences $s^{\rho^{*}}, t^{\rho^{*}}$, we denote by `$s*t$' the concatenation of $s$ and $t$, i.e.\ $(s*t)(i)=s(i)$ for $i<|s|$ and $(s*t)(j)=t(|s|-j)$ for $|s|\leq j< |s|+|t|$. For a sequence $s^{\rho^{*}}$, we define $\overline{s}N:=\langle s(0), s(1), \dots,  s(N)\rangle $ for $N^{0}<|s|$.  
For a sequence $\alpha^{0\di \rho}$, we also write $\overline{\alpha}N=\langle \alpha(0), \alpha(1),\dots, \alpha(N)\rangle$ for \emph{any} $N^{0}$.  By way of shorthand, $q^{\rho}\in Q^{\rho^{*}}$ abbreviates $(\exists i<|Q|)(Q(i)=_{\rho}q)$.  Finally, we shall use $\underline{x}, \underline{y},\underline{t}, \dots$ as short for tuples $x_{0}^{\sigma_{0}}, \dots x_{k}^{\sigma_{k}}$ of possibly different type $\sigma_{i}$.          
\end{nota}
Finally, we discuss an alternative way of formulating G\"odel's system $\T$. 
\begin{rem}[Alternatives to primitive recursion]\rm
 In \cite{escaleert}*{Cor.\ 20}, it is shown that $\textsf{T}$ can be \emph{equivalently} defined using \emph{finite product of selection functions} operators $\pmb{P}^{\rho}$, rather than the usual recursor constants $\pmb{R}^{\rho}$ as in \eqref{PR}.  The definition of $\pmb{P}^{\rho}$ is rather complicated, and therefore omitted, but we do point out that 
in the proof of \cite{escaleert}*{Theorem 18}, the operator $\pmb{P}^{\rho}$ is defined \emph{explicitly} in terms of the recursor constants $\pmb{R}^{\rho}$.  
In other words, $\pmb{P}^{\rho}$ is primitive recursive in the sense of G\"odel's $\T$.   

\medskip
    
Intuitively speaking, the operator $\pmb{P}^{\rho}$ is such that $\pmb{P}^{\rho}(i,m, \dots)$ is defined in terms of $\pmb{P}^{\rho}(i+1,m, \dots)$ for $i\leq m$, and the constant-zero-functional if $i>m$.  In other words, $\pmb{P}^{\rho}$ can call itself, but only $m$ many times before defaulting to a fixed output, namely the constant-zero-functional.  
In light of the definition of $\pmb{P}^{\rho}$ in \cite{escaleert}*{Def.\ 10}, it is immediate that $G$ from \eqref{small} can be expressed in terms of the operator $\pmb{P}^{\rho}$, and $G$ is therefore primitive recursive (in the sense of system $\T$).  Indeed, $G(Y, s, M)$ can only call itself at most $M$ times before defaulting to the first case of \eqref{small}:  Since $G(Y, s, M)$ is defined in terms of $G(Y, s*(n_{1}+1), M)$, which is defined in terms of $G(Y, s*(n_{1}+1)*(n_{2}+1), M)$, and so on, we have that: $G(Y, s*(n_{1}+1)*(n_{2}+1)*\dots* (n_{M}+1), M)=Y(s*(n_{1}+1)*(n_{2}+1)*\dots *(n_{M}+1)*00\dots)$,  which is the first case of \eqref{small}, i.e. $G(Y, s, M)$ defaults to the first case in \eqref{small} after $M$ applications of its definition.  
\end{rem}

\subsection{The classical system $\P$}\label{PIPI}
In this section, we introduce the system $\P$.  We first discuss some of the external axioms studied in \cite{brie}.  

\medskip

Firstly, as in \cite{brie}*{Def.\ 6.1}, we have the following definition of $ \textsf{E-PA}^{\omega*}_{\st} $.  
The language of the latter is the language of $\textsf{E-PA}^{\omega*}$ extended with a new predicate `$\st_{\sigma}$' for every finite type $\sigma \in \pmb{T}^{*}$.  Just as in \cite{brie}, the typing is omitted.
\begin{definition}[The system  $ \textsf{E-PA}^{\omega*}_{\st} $]\label{debs}
The set $\mathcal{T}^{*}$ is defined as the collection of all the constants in the language of $\textsf{E-PA}^{\omega*}$.  
The system $ \textsf{E-PA}^{\omega*}_{\st} $ is defined as $ \textsf{E-PA}^{\omega{*}} + \mathcal{T}^{*}_{\st} + \textsf{IA}^{\st}$, where $\mathcal{T}^{*}_{\st}$
consists of the following axiom schemas:
\begin{enumerate}
\item The schema
 $\st(x)\wedge x=y\di\st(y)$,
\item The schema providing for each closed term $t\in \mathcal{T}^{*}$ the axiom $\st(t)$.
\item The schema $\st(f)\wedge \st(x)\di \st(f(x))$.
\end{enumerate}
The external induction axiom \textsf{IA}$^{\st}$ is as follows, for any formula $\Phi$:  
\be\tag{\textsf{IA}$^{\st}$}
\big(\Phi(0)\wedge(\forall^{\st}n^{0})(\Phi(n) \di\Phi(n+1))\big)\di(\forall^{\st}m^{0})\Phi(m).     
\ee
\end{definition}
Secondly, to guarantee that $\P$ be a conservative extension of Peano arithmetic, 
Nelson's axiom \emph{Standard part} needs to be weakened to $\HAC_{\INT}$ as follows, for any internal formula $\varphi$:  
\be\tag{$\HAC_{\INT}$}
(\forall^{\st}x^{\rho})(\exists^{\st}y^{\tau})\varphi(x, y)\di (\exists^{\st}F^{\rho\di \tau^{*}})(\forall^{\st}x^{\rho})(\exists y^{\tau}\in F(x))\varphi(x,y),
\ee
Note that $F(x)$ provides a \emph{finite sequence} of witnesses to $(\exists^{\st}y)$, explaining its name \emph{Herbrandized Axiom of Choice}.      

\medskip

Thirdly,  Nelson's axiom \emph{Idealisation} requires no weakening and appears in \cite{brie} as follows:  
\be\tag{\textsf{I}}
(\forall^{\st} x^{\sigma^{*}})(\exists y^{\tau} )(\forall z^{\sigma}\in x)\varphi(z,y)\di (\exists y^{\tau})(\forall^{\st} x^{\sigma})\varphi(x,y), 
\ee
where $\varphi$ is again an internal formula.  Intuitively speaking, (the contraposition of) idealisation $\textsf{I}$ allows us to `push all standard quantifiers to the front'.  

\medskip

Fourth, we introduce the system $\P$, for which we need a fragment of the axiom of choice.
Note that $\P$ does not include any fragment of \emph{Transfer}, as the latter translates to non-constructive axioms by Section \ref{seqtor}.
\begin{definition}[The classical system $\P$]~
\begin{enumerate}
\item For internal and quantifier-free $\varphi_{0}$, we have
\be\tag{$\QFAC^{\sigma, \tau}$}
(\forall x^{\sigma})(\exists y^{\tau})\varphi_{0}(x, y)\di (\exists F^{\sigma\di \tau})(\forall x^{\sigma})\varphi_{0}(x, F(x)).
\ee
\item The system $\P$ is defined as $ \textsf{E-PA}^{\omega*}_{\st} +\QFAC^{1,0}+\HAC_{\INT} +\textsf{I}+\textsf{IA}^{\st}$.  
\end{enumerate}
\end{definition}
The system $\P$ is a conservative extension of Peano arithmetic; in particular $\P$ is connected to $\textsf{E-PA}^{\omega*}$ by Theorem~\ref{consresult}.  
The superscript `$S_{\st}$' in the latter is the syntactic translation from \cite{brie}*{Def.\ 7.1} and is defined as:
\begin{definition}\label{corck}
Assume $\Phi(\underline{a})$ and $\Psi(\underline{b})$ in the language of $\P$ have interpretations
\be\label{dombu}
\Phi(\underline{a})^{S_{\st}}\equiv (\forall^{\st}\underline{x})(\exists^{\st}\underline{y})\varphi(\underline{x},\underline{y},\underline{a}) \textup{ and } \Psi(\underline{b})^{S_{\st}}\equiv (\forall^{\st}\underline{u})(\exists^{\st}\underline{v})\psi(\underline{u},\underline{v},\underline{b}),
\ee
where $\psi, \varphi$ are internal.  Then we have the folllowing:  
\begin{enumerate}[(i)]
\item $\psi_{0}^{S_{\st}}:=\psi_{0}$ for atomic internal $\psi_{0}$.  
\item$ \big(\st(z)\big)^{S_{\st}}:=(\exists^{\st}x)(z=x)$.
\item $(\neg \Phi)^{S_{\st}}:=(\forall^{\st} \underline{Y})(\exists^{\st}\underline{x})(\forall \underline{y}\in \underline{Y}[\underline{x}])\neg\varphi(\underline{x},\underline{y},\underline{a})$.  
\item$(\Phi\vee \Psi)^{S_{\st}}:=(\forall^{\st}\underline{x},\underline{u})(\exists^{\st}\underline{y}, \underline{v})[\varphi(\underline{x},\underline{y},\underline{a})\vee \psi(\underline{u},\underline{v},\underline{b})]$.
\item $\big( (\forall z)\Phi \big)^{S_{\st}}:=(\forall^{\st}\underline{x})(\exists^{\st}\underline{y})(\forall z)(\exists \underline{y}'\in \underline{y})\varphi(\underline{x},\underline{y}',z)$.
\end{enumerate}
\end{definition}
\begin{theorem}\label{consresult}
Let $\Phi(\tup a)$ be a formula in the language of \textup{\textsf{E-PA}}$^{\omega*}_{\st}$ and suppose $\Phi(\tup a)^\Sh\equiv\forallst \tup x \, \existsst \tup y \, \varphi(\tup x, \tup y, \tup a)$. If $\Delta_{\intern}$ is a collection of internal formulas and
\be\label{antecedn}
\P + \Delta_{\intern} \vdash \Phi(\tup a), 
\ee
then one can extract from the proof a sequence of closed terms $t$ in $\mathcal{T}^{*}$ such that
\be\label{consequalty}
\textup{\textsf{E-PA}}^{\omega*} +\QFAC^{1,0}+ \Delta_{\intern} \vdash\  \forall \tup x \, \exists \tup y\in \tup t(\tup x)\ \varphi(\tup x,\tup y, \tup a).
\ee
\end{theorem}
\begin{proof}
Immediate by \cite{brie}*{Theorem 7.7}.  
\end{proof}
The proofs of the soundness theorems in \cite{brie}*{\S5-7} actually provide an \emph{algorithm} to obtain the term $t$ from the theorem.  
In other words, one can just `read off' the term $t$ from the proof mentioned in \eqref{antecedn}.  

\medskip

The following corollary is only mentioned in \cite{brie} for Heyting arithmetic, but it also turns out to be valid for Peano arithmetic.  
The proof of the corollary takes place in the same meta-theory as Theorem \ref{consresult}.    
\begin{cor}[Term extraction]\label{consresultcor}
For internal $\psi$ and $\Phi(\underline{a})\equiv(\forall^{\st}\underline{x})(\exists^{\st}\underline{y})\psi(\underline{x},\underline{y}, \underline{a})$, we have $[\Phi(\underline{a})]^{S_{\st}}\equiv \Phi(\underline{a})$.  Hence, if $\Delta_{\intern}$ is a collection of internal formulas and
\[
\P + \Delta_{\intern} \vdash (\forall^{\st}\underline{x})(\exists^{\st}\underline{y})\psi(\underline{x},\underline{y}, \underline{a}), 
\]
then one can extract from the proof a sequence of closed terms $t$ in $\mathcal{T}^{*}$ such that
\[
\textup{\textsf{E-PA}}^{\omega*} +\QFAC^{1,0}+ \Delta_{\intern} \vdash (\forall \underline{x})(\exists \underline{y}\in t(\underline{x}))\psi(\underline{x},\underline{y},\underline{a}).
\]  
\end{cor}
\begin{proof}  
A tedious but straightforward verification using (i)-(v) from Definition \ref{corck} establishes that $[\Phi(\underline{a})]^{S_{\st}}\equiv \Phi(\underline{a})$ for $\Phi(\underline{a})\equiv(\forall^{\st}\underline{x})(\exists^{\st}\underline{y})\psi(\underline{x},\underline{y}, \underline{a})$ and internal $\psi$.  
This verification may also be found in \cite{samzoo}*{\S2} and \cite{sambon}*{\S2.1}.    
\end{proof}
With regard to notation, for the rest of this paper, a \emph{normal form} refers to a formula of the form $(\forall^{\st}x)(\exists^{\st}y)\varphi(x,y)$ for $\varphi$ internal.  Thus, one can say that normal forms are `invariant under $S_{\st}$' in the sense of the previous corollary.    

\medskip

Finally, the previous corollary is central to this paper.  Indeed, a large number of theorems in Nonstandard Analysis can be brought into a normal form (See also Remark \ref{firliborn}), and therefore fall within the scope of Corollary~\ref{consresultcor}. 

\subsection{Notations}\label{tempsde}
We finish this section with remarks on notation in $\P$.  First of all, we mostly follow Nelson, as sketched now.      
\begin{rem}[Nonstandard Analysis]\label{notawin}\rm
We write $(\forall^{\st}x^{\tau})\Phi(x^{\tau})$ and $(\exists^{\st}x^{\sigma})\Psi(x^{\sigma})$ for 
$(\forall x^{\tau})\big[\st(x^{\tau})\di \Phi(x^{\tau})\big]$ and $(\exists x^{\sigma})\big[\st(x^{\sigma})\wedge \Psi(x^{\sigma})\big]$.     
We write $(\forall x\in \Omega)\Phi(x^{0})$ and $(\exists x\in \Omega)\Psi(x^{0})$ as \emph{symbolic}\footnote{As suggested by its name, Nelson's \emph{internal} set theory deals with \emph{internal sets}, i.e.\ sets can only be formed in $\IST$ from \emph{internal} formulas.  In particular, \emph{external} formulas cannot be used to define sets in $\IST$, and a violation of this rule is called \emph{illegal set formation} by Nelson (See \cite{wownelly}).  Thus, our use of `$x\in \Omega$' is purely symbolic, as there is no set of all nonstandard numbers in $\IST$.}  abbreviations for 
$(\forall x^{0})\big[\neg\st(x^{0})\di \Phi(x^{0})\big]$ and $(\exists x^{0})\big[\neg\st(x^{0})\wedge \Psi(x^{0})\big]$.  Furthermore, we write `$x^{0}\in \Omega$' for $\neg\st(x)$.  
A formula $A$ is `internal' if it does not involve `$\st$'; $A^{\st}$ is defined from $A$ by appending `st' to all quantifiers (except bounded number quantifiers).    
\end{rem}
Secondly, we introduce an `approximate' notion of equality.  
\begin{rem}[Approximate quality]\label{equ}\rm
We define `approximate equality $\approx_{\tau}$' as:
\be\label{aparth2}
[x\approx_{\tau}y] \equiv (\forall^{\st} z_{1}^{\tau_{1}}\dots z_{k}^{\tau_{k}})[xz_{1}\dots z_{k}=_{0}yz_{1}\dots z_{k}]
\ee
if the type $\tau$ is composed as $\tau\equiv(\tau_{1}\di \dots\di \tau_{k}\di 0)$.  
Now, the system $\P$ includes the \emph{axiom of extensionality} \eqref{EXT}, but \emph{not} the following version: 
\be\label{EXTST}  
(\forall^{\st}  x^{\rho},y^{\rho},\varphi^{\rho\di \tau}) \big[x\approx_{\rho} y \di \varphi(x)\approx_{\tau}\varphi(y)   \big].
\ee
which is just \eqref{EXT}$^{\st}$, and we shall refer to \eqref{EXTST} as the \emph{axiom of standard extensionality}.
As noted in \cite{brie}*{p.\ 1973}, \eqref{EXT}$^{\st}$ is problematic and cannot be included in $\P$.  
Finally, we need an explicit version of the axiom of extensionality:
\be\tag{$\EXT(\Xi, Y)$}
(\forall f^{1}, g^{1})(\overline{f}\Xi(f,g)=_{0}\overline{g}\Xi(f,g)\di Y(f)=_{0}Y(g)).
\ee
We say that $\Xi^{2}$ is an \emph{extensionality functional} for the functional $Y^{2}$.   
\end{rem}

Finally, we introduce the following (strictly speaking `abuse of') notation.
\begin{rem}[Set-theoretic notation]\label{slettheory}\rm
As in \cite{brie}, we sometimes use intuitive set-theoretic notation, although $\P$ strictly speaking only involves functionals.  First of all, we assume that `sets of numbers $X^{1}$' are given by their characteristic functions $f^{1}_{X}$, i.e.\ $(\forall x^{0})[x\in X\asa f_{X}(x)=1]$.  

\medskip

Secondly, the notation `$Y^{2}\in C$' means that $Y^{2}$ is continuous on Baire space `as usual' given by \eqref{CB2}.  A formula $(\forall^{\st} Y^{2}\in C)(\dots)$ is thus shorthand for $(\forall Y^{2})\big([\st(Y)\wedge Y\in C]\di \dots\big)$; Note in particular that no mention \emph{whatsoever} of \eqref{CB2}$^{\st}$ is made, or will be made in the rest of this paper.  

\medskip

Thirdly, we sometimes block quantifiers together to save space;  In this way, the formula $\big(\forall (Y^{2}\in C, Z^{2})\in \Psi\big)(\dots)$ for some functional $\Psi^{2^{*}}$, is an abbreviation for 
\[
(\forall Z^{2})(\forall Y^{2})\big( \big[Y\in C \wedge (\exists j<|\Psi|)(\Psi(j)=_{2}Y)\wedge (\exists i<|\Psi|)(\Psi(i)=_{2}Z )\big]\di \dots\big),           
\]
which saves considerable space, as will become clear below.  
\end{rem}

\subsection{Applications of Nonstandard Analysis in Computability}\label{XYZ}
We discuss known applications of Nonstandard Analysis in Computability, in particular \cite{charke, dagsam,norhyp}.  

\medskip

First of all, as suggested by its title, the main goal of \cite{charke} is characterising the continuous functionals in highly elementary \emph{nonstandard} terms in Robinson's semantic approach to Nonstandard Analysis.  
In particular, Normann defines a class $F_{k}$ of so-called finitary operators mapping $F_{k-1}$ into $\N$.  In the nonstandard model, $^{*}F_{k}$ is the corresponding nonstandard extension, consisting of the so-called hyper-finitary functionals, and Normann proves that $\textsf{Ct}(k)$ is isomorphic to the \emph{standard part} of $^{*}F_{k}$.  

\medskip

While Normann's approach is similar in spirit to ours (representing complicated objects via elementary nonstandard ones), his nonstandard proofs do not obviously carry computational content.  In particular, the use of \emph{Standard Part} is problematic, as discussed in the final paragraph of this section.        

\medskip

Secondly, as again suggested by its title, the main goal of \cite{norhyp} is also the characterisation of a certain type structure in terms of elementary nonstandard objects.  
The authors state the following:  
\begin{quote}
The novelty [compared to \cite{charke}] here is that we use a constructive version of hyperfinite functionals and also generalise the method to transfinite types. Many of the results of this paper are constructive, though not the characterisation theorems themselves.  (See \cite{norhyp}*{p.\ 1216})
\end{quote}
Hence, the approach from \cite{norhyp} is again similar in spirit to ours (representing complicated objects via elementary nonstandard ones), but the nonstandard proofs again do not obviously carry computational content.  In particular, the use of \emph{standard extensionality} is problematic, as explicitly mentioned in \cite{norhyp}*{p.\ 1218}.        

\medskip

Thirdly, the author and Dag Normann explore the connection between higher-order computability theory and Nonstandard Analysis in \cite{dagsam}.  
The \emph{special fan functional} $\Theta$ from Section \ref{NWKL} (and related functionals based on $\textsf{WWKL}_{0}$ from \cite{simpson2}*{X}) is shown to have quite `non-standard' computational properties:  
On one hand, no type two functional (including $(\exists^{2})$ and $(S^{2})$) can compute the special fan functional, but both $(\exists^{3})$ and $\MUC$ can.  

\medskip

Thus, the special fan functional exhibits extreme computational hardness compared to its first-order strength, but applying the so-called \textsf{ECF}-translation (See \cite{troelstra1}*{\S2.6}) converts the existence of the former into $\WKL_{0}$.  Higher-type generalisations of $\STP$ (and hence of the special fan functional) give rise to even more extreme computational hardness.  
Furthermore, the combination $\paai+\STP$ allows one to derive $\ATR_{0}$ relative to `st', i.e.\ the proofs in \cite{charke} (and hence \cite{norhyp}) seem to take place in highly non-constructive systems.  In particular, it seems difficult to (directly) recover computational content from the proofs in \cite{charke, norhyp}, in contrast to the proofs in this paper.

\section{Preliminaries}\label{prim}
In this section, we prove some preliminary results needed below.  In particular, we study useful fragments of \emph{Standard Part} and \emph{Transfer} from $\IST$ in Section~\ref{monga}.  
Furthermore, in Section \ref{exbisalal}, we derive a version of so-called bar induction for \emph{external} formulas from these fragments of $\IST$.  

\medskip

The theme of this paper is the extraction of relative computability results from theorems of Nonstandard Analysis.  
To further understanding, we will treat in Corollary~\ref{scruf} a very simple example of this theme.  
We also formulate a template for later term extraction results based on this corollary.  
\subsection{Fragments of Standard Part and Transfer}\label{monga}
In this section, we discuss several useful fragments of \emph{Standard Part} and \emph{Transfer} from $\IST$.  
To this end, we first introduce \emph{underspill} and \emph{overspill}, which will be used a lot below.  Intuitively speaking, overspill and underspill express that no internal formula can capture the `st' predicate exactly.  
  
\begin{theorem}\label{spilling}
The system $\P$ proves \emph{overspill} and \emph{underspill}, i.e.\
\[
(\forall^{\st}x^{\rho})\varphi(x)\di (\exists y^{\rho})\big[\neg\st(y)\wedge \varphi(y)  \big] \textup{ and }(\forall x^{\rho})\big[\neg\st(x)\di \varphi(x)]\di (\exists^{\st} y^{\rho})\varphi(y),
\]
for any internal formula $\varphi$.
\end{theorem}
\begin{proof}
Immediate by \cite{brie}*{Prop.\ 3.3 and \S5}.  
\end{proof}
We apply underspill most frequently as follows:  From $(\forall M\in \Omega)\psi(M)$ for internal $\psi$, we conclude $(\forall K^{0})\big[\neg\st(K)\di (\forall M\geq K)\psi(M) \big]$.  Applying underspill for $\varphi(K)\equiv (\forall M\geq K)\psi(M)$, we obtain $(\exists^{\st} K^{0}) (\forall M\geq K)\psi(M)$.   

\subsubsection{The nonstandard counterpart of weak K\"onig's lemma}\label{NWKL}
In this section, we study the following fragment of the \emph{Standard part} principle of $\IST$:
\be\label{STP}\tag{\textup{\textsf{STP}}}
(\forall f^{1}\leq_{1}1)(\exists^{\st} g^{1}\leq_{1}1)(f\approx_{1}g).
\ee   
The function $g^{1}$ from \ref{STP} is called a \emph{standard part} of $f^{1}$.  
By the following theorem, $\STP$ is a nonstandard version of \emph{weak K\"onig's lemma} ($\WKL$).  The latter is the statement that a infinite binary tree has a path (See e.g.\ \cite{simpson2}*{IV}).  
\begin{theorem}\label{agda}
The system $\P$ proves that \ref{STP} is equivalent to 
\begin{align}\label{fanns}
(\forall T^{1}\leq_{1}1)\big[(\forall^{\st}n)(\exists \beta^{0})(|\beta|=n \wedge \beta\in T ) 
\di (\exists^{\st}\alpha^{1}\leq_{1}1)(\forall^{\st}n^{0})(\overline{\alpha}n\in T)   \big].
\end{align}
where ` $T\leq_{1}1$' means that $T$ is a binary tree.  Over $\P$, \ref{STP} is also equivalent to    
\be\label{krog}
(\forall f^{1})(\exists^{\st} g^{1})\big( (\forall^{\st}n^{0})(\exists^{\st}m^{0})(f(n)=m)\di   f\approx_{1}g\big).
\ee
\end{theorem}
\begin{proof}    
Assume \ref{STP} and apply overspill to $(\forall^{\st}n)(\exists \beta^{0})(|\beta|=n \wedge \beta\in T )$ to obtain $\beta_{0}^{0}\in T$ with nonstandard length $|\beta_{0}|$.  
Now apply \ref{STP} to $\beta^{1}:=\beta_{0}*00\dots$ to obtain a \emph{standard} $\alpha^{1}\leq_{1}1$ such that $\alpha\approx_{1}\beta$ and hence $(\forall^{\st}n)(\overline{\alpha}n\in T)$.  
For the reverse direction, let $f^{1}$ be a binary sequence, and define a binary tree $T_{f}$ which contains all initial segments of $f$.  
Now apply \eqref{fanns} for $T=T_{f}$ to obtain \ref{STP}.  

\medskip

For the final equivalence, $\eqref{krog}\di \ref{STP}$ is trivial, and for the reverse implication, fix $f^{1}$ such that 
$(\forall^{\st}n)(\exists^{\st}m)f(n)=m$ and let $h^{1}$ be such that $(\forall n,m)(f(n)=m\asa h(n,m)=1)$.  
Applying $\HAC_{\INT}$ to the former, there is standard $\Phi^{0\di 0^{*}}$ such that $(\forall^{\st}n)(\exists m\in \Phi(n))f(n)=m$, and define 
$\Psi(n):=\max_{i<|\Phi(n)|}\Phi(n)(i)$.  Now define $\alpha_{0}\leq_{1}1$ as:  $\alpha_{0}(0):= h(0,0)$, $\alpha_{0}(1):=h(0, 1)$, \dots, $\alpha_{0}(\Psi(0)):=h(0,\Psi(0))$, $\alpha_{0}(\Psi(0)+1):= h(1, 0)$, $\alpha_{0}(\Psi(0)+2):= h(1, 1)$, \dots, $\alpha_{0}(\Psi(0)+\Psi(1)):= h(1, \Psi(1))$, et cetera.  
Now let $\beta_{0}^{1}\leq_{1}1$ be the standard part of $\alpha_{0}$ provided by $\ref{STP}$ and define $g(n):=(\mu m\leq \Psi(n))\big[\beta_{0}(\sum_{i=0}^{n-1}\Psi(i)+m)=1\big]$.  By definition, $g^{1}$ is standard and $f\approx_{1} g$.     
\end{proof}    
The function $g^{1}$ from \eqref{krog} is also called a \emph{standard part} of $f^{1}$.  
We now show that \ref{STP} follows from the nonstandard uniform continuity of all type two functionals on Cantor space.  
Note that the principle \ref{kunt} in the theorem contradicts classical mathematics, as the latter involves \emph{discontinuous} functionals.    
\begin{theorem}\label{foor}
The axiom \ref{STP} can be proved in $\P$ plus the axiom
\be\label{kunt}\tag{\textsf{\textup{NUC}}}
(\forall^{\st}Y^{2})(\forall f^{1}, g^{1}\leq_{1}1)(f\approx_{1} g\di Y(f)=_{0}Y(g)).  
\ee
\end{theorem}
\begin{proof}
First of all, note that \ref{kunt} implies by Remark \ref{equ} that
\be\label{kunt2}
(\forall^{\st}Y^{2})(\forall f^{1}, g^{1}\leq_{1}1)(\exists^{\st}N^{0})(\overline{f}N=_{0} \overline{g}N\di Y(f)=_{0}Y(g)).  
\ee
Applying idealisation \textsf{I} to \eqref{kunt2}, we obtain that 
\be\label{kunt32}
(\forall^{\st}Y^{2})(\exists^{\st}x^{0^{*}})(\forall f^{1}, g^{1}\leq_{1}1)(\exists N^{0}\in x)(\overline{f}N=_{0} \overline{g}N\di Y(f)=_{0}Y(g)).  
\ee
which immediately yields that
\be\label{kunt3}
(\forall^{\st}Y^{2})(\exists^{\st}N_{0}^{0})(\forall f^{1}, g^{1}\leq_{1}1)(\overline{f}N_{0}=_{0} \overline{g}N_{0}\di Y(f)=_{0}Y(g)),   
\ee
by taking $N_{0}$ in \eqref{kunt3} to be $\max_{i<|x|}x(i)$ for $x$ as in \eqref{kunt32}.  
However, this implies that for $Y$ and $N_{0}$ as in \eqref{kunt3}, we have
\be\textstyle\label{fruck}
 (\forall f^{1}\leq_{1}1)(Y(f)\leq \max_{\sigma\leq_{0^{*}}1 \wedge |\sigma|=N_{0}})Y(\sigma*00\dots), 
\ee
i.e.\ $Y$ attains a standard maximum on Cantor space.  
In this light, consider the contraposition of \eqref{fanns} for some fixed $T\leq_{1}1$, and assume $(\forall^{\st}\alpha^{1}\leq_{1}1)(\exists^{\st}n^{0})(\overline{\alpha}n\not\in T)$.  
Applying $\HAC_{\INT}$ yields a standard functional $Y_{0}^{1\di 0^{*}}$ such that $(\forall^{\st}\alpha^{1}\leq_{1}1)(\exists i\in Y_{0}(\alpha))(\overline{\alpha}i\not\in T)$.  Now define $Y_{1}^{2}$ by $Y_{1}(\alpha):=\max_{i<|Y_{0}(\alpha)|}Y_{0}(\alpha)(i)$ and note that $(\forall^{\st}\alpha^{1}\leq_{1}1)(\exists n^{0}\leq Y_{1}(\alpha))(\overline{\alpha}n\not\in T)$ by definition.  
By the previous, $Y_{1}$ has a standard upper bound on Cantor space as in \eqref{fruck}, yielding $(\exists^{\st}k^{0})(\forall \beta^{1}\leq_{1}1)(\exists i\leq k)(\overline{\beta}i\not\in T)$, and \ref{STP} follows from Theorem~\ref{agda}.        
\end{proof}
The following is a more direct proof of Theorem \ref{foor}, not requiring Theorem \ref{agda}.
\begin{proof}
Working in $\P+\NUC$, suppose there is some $g_{0}\leq_{1}1$ such that $(\forall^{\st}f^{1}\leq_{1}1)(f\not\approx_{1}g_{0})$.  
The latter implies $(\forall^{\st}f^{1}\leq_{1}1)(\exists^{\st}n^{0})(\overline{f}n\ne\overline{g_{0}}n)$, 
and applying $\HAC_{\INT}$ yields \emph{standard} $Y_{0}^{1\di 0^{*}}$ such that $(\forall^{\st}f^{1}\leq_{1}1)(\exists n^{0}\in Y_{0}(f))(\overline{f}n\ne\overline{g_{0}}n)$.  
Define standard $Y_{1}^{2}$ by $Y_{1}(f):=\max_{i<|Y_{0}(f)|}Y_{0}(f)(i)$ and note that $(\forall^{\st}f^{1}\leq_{1}1)(\overline{f}Y_{1}(f)\ne\overline{g_{0}}Y_{1}(f))$.  
By the above, $\NUC\di \eqref{kunt3}$, which yields $(\forall^{\st}Y^{2})(\forall g^{1}\leq_{1}1)(\exists^{\st}K^{0})(Y(g)\leq K)$ by taking $K=Y(\overline{g}N_{0}*00\dots)$ for $N_{0}$ as in \eqref{kunt3}.  Note that  $\overline{g}N_{0}$ is standard by \cite{brie}*{Cor.\ 2.19}.
Applying \emph{idealisation} \textsf{I}, we obtain the formula $(\forall^{\st}Y^{2})(\exists^{\st} k^{0^{*}})(\forall g^{1}\leq_{1}1)(\exists K^{0}\in k)(Y(g)\leq K)$, and the maximum of $k$ yields the existence of a standard upper bound for any standard $Y^{2}$ on Cantor space.  In particular, $Y_{1}$ has a standard upper bound on Cantor space, say $m_{1}$.  
Then define the sequence $g^{1}_{1}$ as $\overline{g_{0}}m_{1}*00\dots$, 
and note that $\overline{g_{0}}m_{1}$ is standard by \cite{brie}*{Cor.\ 2.19}, and hence $g_{1}$ is also standard.  By the definition of $Y_{1}$, we have $\overline{g_{1}}Y_{1}(g_{1})\ne \overline{g_{0}}Y_{1}(g_{1})$, which yields a contradiction.     
\end{proof}
As explained in the introduction, the theme of this paper is the extraction of relative computability results from theorems of Nonstandard Analysis.  
We now provide the first example of this theme in Corollary \ref{scruf}, based on the proof of $\NUC\di \STP$ in the previous theorem.  
The following definitions are relevant.  
\be
(\forall Y^{2}) (\forall f^{1}, g^{1}\leq_{1}1)(\overline{f}\Omega(Y)=\overline{g}\Omega(Y)\notag \di Y(f)=Y(g)). \label{lukl3}\tag{$\textsf{\textup{MUC}}(\Omega)$}
\ee
\begin{align}\label{fanns333}\tag{$\SCF(\Theta)$}
(\forall g^{2}, T^{1}\leq_{1}1)\big[ (\forall  \alpha^{1}\in \Theta(g)(2))(\alpha\leq_{1}1\di \overline{\alpha}g(\alpha)\not\in T)\di(\forall \beta\leq_{1}1)(\exists i\leq_{0}\Theta(g)(1))(\overline{\beta}i\not\in T) \big].
\end{align}
The functional $\Omega^{3}$ as in $\MUC(\Omega)$ is the (intuitionistic) \emph{fan functional} and yields a conservative extension of $\WKL$ for the second-order language (See \cite{kohlenbach2}*{Prop.\ 3.15}).  By Corollary \ref{scruf}, $\Theta$ as in $\SCF(\Theta)$ is a `special case' of $\Omega$ and we refer to $\Theta$ as the \emph{special} fan functional (although $\Theta$ is strictly speaking not unique).  

\medskip

The computational properties of $\Theta$ have been studied in \cite{dagsam} and are briefly sketched in Section \ref{XYZ}.  From a computability theoretic perspective, the main property of $\Theta$ is the selection of $\Theta(g)(2)$ as a finite sequence of binary sequences $\langle f_0 , \dots, f_n\rangle $ such that the neighbourhoods defined from $\overline{f_i}g(f_i)$ for $i\leq n$ form a cover of Cantor space;  almost as a by-product, $\Theta(g)(1)$ can then be chosen to be the maximal value of $g(f_i) + 1$ for $i\leq n$. We stress that $g^{2}$ in $\SCF(\Theta)$ may be \emph{discontinuous} and that Kohlenbach has argued for the study of discontinuous functionals in higher-order RM (See \cite{kohlenbach2}*{\S1}). 
In the absence of discontinuous functionals, $\Theta$ behaves as follows.  
\begin{cor}\label{scruf}
From the proof in $\P$ that $\NUC\di \STP$, a term $t^{3\di3}$ can be extracted such that 
$\textup{\textsf{E-PA}}^{\omega*}+\QFAC^{1,0}$ proves $(\forall \Omega^{3})\big[\MUC(\Omega)\di \SCF(t(\Omega))]$.  
\end{cor}
\begin{proof}
By the proof of the theorem, $\NUC$ is equivalent to the normal form \eqref{kunt3}, which we abbreviate as $(\forall^{\st}Y^{2})(\exists^{\st}N^{0})A(Y,N)$.  
The contraposition of \eqref{fanns} is 
\begin{align}\label{fannsXXY}
(\forall T^{1}\leq_{1}1)\big[ (\forall^{\st}\alpha\leq_{1}1)(\exists^{\st}n^{0})(\overline{\alpha}n\not\in T)\di
 (\exists^{\st}k^{0})(\forall \beta\leq_{1}1)(\exists i\leq k)(\overline{\beta}i\not\in T) \big].
\end{align}
Since standard functionals have standard output for standard input, \eqref{fannsXXY} implies:
\begin{align}\label{fannsXXX}
(\forall T^{1}\leq_{1}1)(\forall^{\st}g^{2})\big[ (\forall^{\st}\alpha\leq_{1}1)(\overline{\alpha}g(\alpha)\not\in T)\di
 (\exists^{\st}k^{0})(\forall \beta\leq_{1}1)(\exists i\leq k)(\overline{\beta}i\not\in T) \big].
\end{align}
Pushing all standard quantifiers outside, we obtain that
\[
(\forall^{\st}g^{2})(\forall T^{1}\leq_{1}1)(\exists^{\st}k^{0}, \alpha^{1}\leq_{1}1)\big[(\overline{\alpha}g(\alpha)\not\in T)\di(\forall \beta\leq_{1}1)(\exists i\leq k)(\overline{\beta}i\not\in T) \big].
\]
Applying idealisation $\textsf{I}$, we pull the standard quantifiers to the front as follows:
\begin{align}\label{teringzooi}
(\forall^{\st}g^{2})(\exists^{\st}w^{1^{*}})(\forall T^{1}\leq_{1}1)(\exists ( \alpha^{1}\leq_{1}1,  k^{0}) \in w)\big[(\overline{\alpha}g(\alpha)\not\in T)
\di(\forall \beta\leq_{1}1)(\exists i\leq k)(\overline{\beta}i\not\in T) \big], 
\end{align}
which we abbreviate as $(\forall^{\st}g^{2})(\exists^{\st}w^{1^{*}})B(g, w)$.  Hence, the proof of $\ref{kunt}\di \STP$ yields a proof of 
$(\forall^{\st}Y^{2})(\exists^{\st}N^{0})A(Y,N)\di (\forall^{\st}g^{2})(\exists^{\st}w^{1^{*}})B(g, w)$, which yields
\be\label{hoer}
\big[(\exists^{\st}\Omega^{3})(\forall Y^{2})A(Y,\Omega(Y))\di (\forall^{\st}g^{2})(\exists^{\st}w^{1^{*}})B(g, w)\big], 
\ee
by strengthening the antecedent.  Bringing all standard quantifiers up front:
\be\label{calvarie}
(\forall^{\st}\Omega^{3}, g^{2})(\exists^{\st}w^{1^{*}})\big[(\forall Y^{2})A(Y,\Omega(Y))\di B(g, w)\big];
\ee
Applying Corollary \ref{consresultcor} to `$\P\vdash \eqref{calvarie}$', we obtain a term $t^{3\di3}$ such that 
\be\label{calvarie2}
(\forall \Omega^{3}, g^{2})(\exists w\in t(\Omega, g))\big[(\forall Y^{2})A(Y,\Omega(Y))\di B(g, w)\big] 
\ee
is provable in $\textsf{E-PA}^{\omega*}+\QFAC^{1,0}$.  
Bringing all quantifiers inside again, \eqref{calvarie2} yields
\be\label{calvarie3}
(\forall \Omega^{3})\big[(\forall Y^{2})A(Y,\Omega(Y))\di  (\forall g^{2})(\exists w\in t(\Omega, g))B(g, w)\big], 
\ee
Clearly, the antecedent of \eqref{calvarie3} expresses that $\Omega^{3}$ is the fan functional.  To define a functional $\Theta$ as in $\SCF(\Theta)$ from $t(\Omega, g)$, note that the latter is a finite sequence of numbers and binary sequences by \eqref{teringzooi};  Using basic sequence coding, we may assume that $t(\Omega, g)=t_{0}(\Omega, g)*t_{1}(\Omega, g)$, where the first (resp.\ second) part contains the binary sequences (resp.\ numbers).  Now define $\Theta( g)(1)$ as $\max_{i<|t_{1}(\Omega, g)|}t_{1}(\Omega, g)(i)$ and $\Theta( g)(2):=t_{0}(\Omega, g)$, and note that $\Theta$ indeed satisfies $\SCF(\Theta)$.   
\end{proof}
In the following remark, we discuss how the proof of Corollary \ref{scruf} provides a template for the rest of the term extraction results in this paper.  
We will apply this template to another basic example in Section \ref{seqtor}.
\begin{rem}[Template for term extraction]\label{firliborn}\rm
~
\begin{enumerate}[(i)]
\item Bring antecedent and consequent in normal form (See \eqref{kunt3} and \eqref{teringzooi}).
\item Introduce a \emph{standard} witnessing functional in the antecedent, and drop the remaining `st' (See \eqref{hoer}). \label{ted}
\item Bring all standard quantifiers up front to obtain a normal form (See \eqref{calvarie}).  
Use idealisation \textsf{I} when encountering `$(\forall x)(\exists^{\st}y)$' (Irrelevant for Corollary \ref{scruf}).\label{cruz}
\item Apply Corollary \ref{consresultcor} to the proof in $\P$ of the normal form to obtain a term $t$ and a proof in \textsf{E-PA}$^{\omega*}$ (See \eqref{calvarie2}).  
\item Bring all quantifiers inside again to obtain the sought-after relative computability result (See \eqref{calvarie3}).
\end{enumerate}
All further term extraction results follow this template, but often in less detail.  It is interesting to note that the \emph{nonstandard} definitions of e.g.\ continuity, Riemann integration, compactness, differentiability, etc, have normal forms in (fragments of) $\P$;  Furthermore, normal forms are `closed under modus ponens' in the sense that an implication between two normal forms can be brought into a normal form too (As is done in items \eqref{ted} and \eqref{cruz} above).  Thus, the template seems to apply to any theorem which only involves nonstandard definitions.  This is explored in the context of Reverse Mathematics and its `zoo' in \cite{samzoo,sambon}.  
\end{rem}
\subsubsection{Nonstandard arithmetical comprehension}\label{seqtor}
In this section, we apply the template from Remark \ref{firliborn} to another instructive example, involving the following fragment of Nelson's axiom \emph{Transfer}:  
\be\tag{\textup{$\sigtoe$}}\label{predruk}
(\forall^{\st}f^{1})\big[(\exists m^{0})(\forall n^{0})f(m,n)=0\di (\exists^{\st} k^{0})(\forall l^{0})f(k,l)=0\big].
\ee
and the following \emph{nonstandard continuity principle}:
\be\label{druk22}\tag{\textup{\textsf{NPC}}}
(\forall^{\st}Y^{2}\in C, f^{1})(\forall g^{1})(f\approx_{1}g\di Y(f)=_{0}Y(g)).
\ee
By Corollary \ref{seconde}, $\sigtoe$ is a nonstandard `precursor' to arithmetical comprehension.  We have the following theorem.    
\begin{theorem}\label{loeder}
The system $\P$ proves that $\sigtoe\di \NPC$.  
\end{theorem}
\begin{proof}
We make essential use of the proof of \cite{kohlenbach4}*{Prop.\ 4.7}.  
In the latter, it is shown that a continuous functional on Baire space has a modulus of continuity, assuming arithmetical comprehension.  
With our notations and for $Y^{2}\in C$, the aforementioned proof amounts to weakening \eqref{CB2} to:  
\be\label{karaffe}
(\forall f^{1})(\exists N^{0})\big[(\forall \tau^{0^{*}}, \sigma^{0^{*}})( Y(\overline{f}N*\sigma*00\dots)=_{0}Y(\overline{f}N*\tau*00\dots))\big],   
\ee
and using arithmetical comprehension to obtain the characteristic function $\chi^{2}$ of the formula in square brackets in \eqref{karaffe}.    
The latter then yields $(\forall f^{1})(\exists N^{0})(\chi(f,N)=1)$, and $\QFAC^{1,0}$ and quantifier-free induction are applied to the former to obtain $H^{2}$
such that $H(f)$ is the least such $N$.  The functional $H^{2}$ is then shown to be the modulus of continuity for $Y^{2}\in C$.  

\medskip

Applying $\sigtoe$ to \eqref{karaffe} for standard $f^{1}$ and standard $Y^{2}\in C$, yields:
\[
(\forall^{\st} f^{1})(\exists^{\st} N^{0})\big[(\forall \tau^{0^{*}}, \sigma^{0^{*}})( Y(\overline{f}N*\sigma*00\dots)=_{0}Y(\overline{f}N*\tau*00\dots))\big].   
\]
Hence, by the leastness of $H(f)$ from the previous paragraph, the latter is standard for standard $f^{1}$.  Since $H$ is a modulus of pointwise continuity, we thus obtain 
\[
(\forall^{\st} f^{1})(\exists^{\st} N^{0})(\forall g^{1})(\overline{f}N=_{0}\overline{g}N\di Y(f)=_{0}Y(g)), 
\]
which immediately implies that $Y^{2}$ is nonstandard continuous, and $\NPC$ follows.
\end{proof}
To apply term extraction to Theorem \ref{loeder}, the following principles are needed.   
\be
(\forall Y^{2}\in C, f^{1}, g^{1})(\overline{f}\Psi(Y, f)=\overline{g}\Psi(Y, f)\notag \di Y(f)=Y(g)). \label{lukl2}\tag{$\textsf{\textup{MPC}}(\Psi)$}
\ee
\be\tag{$\MU(\mu)$}
(\forall f^{1})\big[(\exists n^{0})f(n)=0\di f(\mu(f))=0\big].
\ee
Note that $\MPC(\Psi)$ states that $\Psi^{3}$ is a modulus-of-continuity functional, while $\MU(\mu)$ states that $\mu^{2}$ is Feferman's search operator (See e.g.\ \cite{kohlenbach2} for the latter).  
One usually abbreviates `$(\exists{\mu^{2}})\MU(\mu)$' by $(\mu^{2})$, and the latter provides arithmetical comprehension in the sense of \eqref{arco} below.  
\begin{cor}\label{seconde}
From the proof in $\P$ that $\sigtoe\di \NPC$, a term $t$ can be extracted such that 
$\textup{\textsf{E-PA}}^{\omega*}+\QFAC^{1,0}$ proves $(\forall \mu^{2})\big[\MU(\mu)\di \MPC(t(\mu))]$.  
\end{cor}
\begin{proof}
First of all, a normal form for $\sigtoe$ is:
\be\label{normaltrank}
(\forall^{\st}f^{1}) (\exists^{\st} k^{0})\big[(\exists m^{0})(\forall n^{0})f(m,n)=0\di(\exists i\leq k)(\forall l)f(i,l)=0\big],
\ee
where $A(f, k)$ is the internal formula in square brackets.  A normal form for $\NPC$ is
\be\label{kuntnor}
(\forall^{\st}Y^{2}\in C, f^{1})(\exists^{\st}N_{0}^{0})\big[(\forall g^{1})(\overline{f}N_{0}=_{0} \overline{g}N_{0}\di Y(f)=_{0}Y(g))\big],   
\ee
which can be obtained in exactly the same way that \eqref{kunt3} is derived from \ref{kunt}.  Let $B(Y, f, N_{0})$ be the formula in square brackets in \eqref{kuntnor}.  The implication 
$\sigtoe\di \NPC$ thus implies 
\be\label{eqref}
(\forall^{\st}f^{1}) (\exists^{\st} k^{0})A(f, n)\di (\forall^{\st}Y^{2}\in C, g^{1})(\exists^{\st}N_{0}^{0})B(Y, g, N_{0}).  
\ee      
As standard functionals yield standard outputs for standard inputs, we may strengthen the antecedent of \eqref{eqref} as follows:
\be\label{eqref2}
(\forall^{\st}\nu^{2})\big[(\forall^{\st}f^{1}) A(f, \nu(f))\di (\forall^{\st}Y^{2}\in C, g^{1})(\exists^{\st}N_{0}^{0})B(Y, g, N_{0})\big].  
\ee
We may also strengthen the antecedent by dropping the `st' to obtain
\be\label{eqref246}
(\forall^{\st}\nu^{2})\big[(\forall f^{1}) A(f, \nu(f))\di (\forall^{\st} Y^{2}\in C, g^{1})(\exists^{\st}N_{0}^{0})B(Y, g, N_{0})\big].  
\ee
Brining all standard quantifiers to the front, we obtain the normal form
\be\label{eqref3}
(\forall^{\st}\nu^{2}, Y^{2}\in C, g^{1})(\exists^{\st}N_{0}^{0})\big[(\forall f^{1}) A(f, \nu(f))\di B(Y, g, N_{0})\big].  
\ee
Apply Corollary \ref{consresultcor} to `$\P\vdash \eqref{eqref3}$' to obtain a term $t$ s.t.\ $\textsf{E-PA}^{\omega*}+\QFAC^{1,0}$ proves
\be\label{eqref4}
(\forall \nu^{2}, Y^{2}\in C, g^{1})(\exists N_{0}^{0}\in t(\nu, Y, g))\big[(\forall f^{1}) A(f, \nu(f))\di B(Y, g, N_{0})\big].  
\ee
Now define the term $s$ by $s(\nu, Y, g ):=\max_{i<|t(\nu, Y, g)|}t(\nu, Y, g)(i)$ and note that 
\be\label{eqref5}
(\forall \nu^{2}, Y^{2}\in C, g^{1})\big[(\forall f^{1}) A(f, \nu(f))\di B(Y, g, s(\nu, Y, g))\big].  
\ee
Bringing all quantifiers inside again, we obtain that 
\be\label{eqref6}
(\forall \nu^{2})\big[(\forall f^{1}) A(f, \nu(f))\di (\forall Y^{2}\in C, g^{1})B(Y, g, s(\nu, Y, g))\big], 
\ee
where the consequent is clearly $\MPC(u(\nu))$ for $u(\nu)(Y, g):=s(\nu, Y,g )$.  

\medskip

Finally, it is easy to define $\nu$ as in $(\forall f^{1})A(f, \nu(f))$ from \eqref{eqref6} explicitly in terms of $\mu^{2}$ as in $\MU(\mu)$;  Indeed, for such $\mu^{2}$ we have: 
\be\label{arco}
(\forall h^{1})\big[ (\forall k^{0})(h(k)\ne0) \asa h(\mu(h))\ne0   \big], 
\ee
i.e.\ $\mu^{2}$ as in $\MU(\mu)$ allows us to decide universal formulas.  
Thus, $(\exists m^{0})(\forall n^{0})f(m,n)=0$ is equivalent to $(\exists m^{0})f\big(m,\mu((\lambda n)\tilde{f}(m,n))\big)=0$, where for $\tilde{k}$ is defined as 
\[
\tilde{k}(n):=
\begin{cases}
1&   k(n)=0 \\
0 & \textup{otherwise}
\end{cases}
\]
Applying the definition of $\mu^{2}$ to $(\exists m^{0})f\big(m,\mu((\lambda n)\tilde{f}(m,n))\big)=0$, we observe that $(\exists m^{0})(\forall n^{0})f(m,n)=0$ yields $(\forall n^{0})f(\mu((\lambda m)\mu((\lambda n)\tilde{f}(m,n))),n)=0$, and $\nu(f):=(\lambda f)\mu((\lambda m)\mu((\lambda n)\tilde{f}(m,n)))$ is as required for $(\forall f^{1})A(f, \nu(f))$, and we are done.  
\end{proof}
As is clear from the last part of the proof, $\sigtoe$ is actually the nonstandard precursor of the search functional:
\be\tag{$\MU_{2}(\nu)$}
(\forall f^{1})\big[(\exists m^{0})(\forall n^{0})f(m,n)=0\di (\forall l^{0})f(\nu(f),l)=0\big],
\ee
but by the previous proof, two applications of Feferman's search operator $\mu$ yields $\nu$ as in $\MU_{2}(\nu)$.  We refer to the latter as `Feferman's \emph{second} search functional', which will be needed in Section \ref{relletje}.

\subsubsection{A `computable' fragment of the Standard Part principle}
In this section, we discuss the \emph{Standard Part} principle $\Omega$\textsf{-CA}, a very practical consequence of \HAC$_{\textup{\INT}}$.  
Intuitively speaking, $\OCA$ expresses that we can obtain the standard part (in casu $G$) of \emph{$\Omega$-invariant} nonstandard objects (in casu $F(\cdot,M)$), defined as follows.     
\begin{definition}[$\Omega$-invariance]\label{homega} Let $F^{(\sigma\times  0)\di 0}$ be standard and fix $M^{0}\in \Omega$.  
Then $F(\cdot,M)$ is {\bf $\Omega$-invariant} if   
\be\label{homegainv}
(\forall^{\st} x^{\sigma})(\forall N^{0}\in \Omega)\big[F(x ,M)=_{0}F(x,N) \big].  
\ee
\end{definition}
\begin{princ}[$\Omega$\textsf{-CA}]\rm Let $F^{(\sigma\times 0)\di 0}$ be standard and fix $M\in \Omega$.
For $\Omega$-invariant $F(\cdot,M)$, there is standard $G^{\sigma\di 0}$ such that
\be\label{homegaca}
(\forall^{\st} x^{\sigma})(\forall N^{0}\in \Omega)\big[G(x)=_{0}F(x,N) \big].  
\ee
\end{princ}
In line with $\STP$, we also refer to $G(\cdot)$ as `a standard part of $F(\cdot, N)$'.  
Intuitively speaking, $\Omega$\textsf{-CA} provides a standard part for a nonstandard object, if the latter is \emph{independent of the choice of nonstandard number} used in its definition.
\begin{theorem}\label{drifh}
The system $\P$ proves $\Omega\textup{-\textsf{CA}}$.  
\end{theorem}
\begin{proof}
Assume $F(\cdot,M^{0})$ is $\Omega$-invariant, i.e.\ we have 
\be\label{korfs}
(\forall^{\st} x^{\sigma})(\forall N^{0},M^{0}\in \Omega)\big[F(x ,M)=_{0}F(x,N) \big],  
\ee
and underspill (See Theorem \ref{spilling}) implies that
\be\label{horf}
(\forall^{\st} x^{\sigma})(\exists^{\st}k^{0})(\forall N^{0},M^{0}\geq k)\big[F(x ,M)=_{0}F(x,N) \big].
\ee
Now apply \HAC$_{\textup{\INT}}$ to \eqref{horf} to obtain standard $\Phi^{\sigma\di 0^{*}}$ such that
\[
(\forall^{\st} x^{\sigma})(\exists k^{0}\in \Phi(x))(\forall N^{0},M^{0}\geq k)\big[F(x ,M)=_{0}F(x,N) \big].
\]
Define standard $\Psi(x):= \max_{i<|\Phi(x)|}\Phi(x)(i)$ and note that 
\be\label{kiru}
(\forall^{\st} x^{\sigma})(\forall N^{0},M^{0}\geq \Psi(x))\big[F(x ,M)=_{0}F(x,N) \big].
\ee
Finally, define $G(x):=F(x,\Psi(x))$ and note that the latter is as in $\OCA$.
\end{proof}

We finish this section with two remarks on the above results.  
\begin{rem}[Extensions of $\OCA$]\label{lop}\rm
It is straightforward to verify that Theorem~\ref{drifh} also holds if the quantifier `$(\forall^{\st} x^{\sigma})$' in \eqref{homegainv} and \eqref{homegaca} is restricted as in `$(\forall^{\st} x^{\sigma})(C(x)\di \dots)$', where $C$ is any \emph{internal} formula.  We shall also refer to this slight extension as $\OCA$.    
The axiom $\OCA$ can also be generalised to $F^{(\sigma\times 0)\di \tau}$ using the approximate equality `$\approx_{\tau}$' defined in Remark \ref{equ}.  
However, the above version suffices for our purposes.  
\end{rem}
\begin{rem}[Using $\HAC_{\INT}$ and $\textsf{I}$]\label{simply}\rm
The axiom $\HAC_{\INT}$ produces a functional of type $\sigma\di \tau^{*}$ which outputs a \emph{finite sequence} of witnesses.  
Now, in the proof of Theorem~\ref{drifh}, $\HAC_{\INT}$ is applied to \eqref{horf} to obtain $\Phi^{\sigma\di 0^{*}}$, and from the latter, the functional $\Psi^{\sigma\di 0}$ is defined as follows: $\Psi(x):= \max_{i<|\Phi(x)|}\Phi(x)(i)$.  In particular, $\Psi$ satisfies \eqref{kiru}, and provides a \emph{witnessing functional}, due to the `monotone' nature of the internal formula in \eqref{horf}.  In general, $\HAC_{\INT}$ provides a \emph{witnessing functional} assuming (i) $\tau=0$ in $\HAC_{\INT}$ and (ii) the formula $\varphi$ from $\HAC_{\INT}$ is `sufficiently monotone' as in: 
$(\forall x^{\sigma},n^{0},m^{0})\big([n\leq_{0}m \wedge\varphi(n,x)] \di \varphi(m,x)\big)$.    

\medskip

A similar observation applies to idealisation \textsf{I};  Indeed, consider \eqref{kunt2} and note that the internal formula in the latter is monotone as above.  Taking the maximum of $x$ from \eqref{kunt32} as $N_{0}:=\max_{i<|x|}x(i)$, one can drop the quantifier $(\exists N\in x)$ in \eqref{kunt32} to obtain \eqref{kunt3}.  
To save space in proofs, we will sometimes skip the (obvious) step involving the maximum of the finite sequences, when applying $\HAC_{\INT}$ and $\textsf{I}$. 
\end{rem}
\subsection{External bar induction}\label{exbisalal}
In this section, we derive various versions of bar induction inside extensions of $\P$ studied in the previous section.
Now, bar induction can be viewed as `induction down a tree', and we consider the following example. 
\begin{princ}[\textsf{BI}$_{0}$]\label{BI0}
For internal quantifier-free $Q(x^{0})$, if
\be
(\forall \alpha^{1} )(\exists  n^{0})Q(\overline{\alpha}n) \wedge  (\forall t^{0})\big[(\forall x^{0})Q(t*\langle x\rangle)\di Q(t)\big]\label{cond2b}
\ee
then we have $Q(\langle\rangle)$.  
\end{princ}
Intuitively speaking, bar induction $\BI_{0}$ expresses that we may conclude $Q(x)$ for $x=\langle\rangle$ from the fact that $Q$ is implied `downwards' from child nodes to parent nodes (second conjunct of \eqref{cond2b}) and that $Q$ holds eventually along any path (first conjunct of \eqref{cond2b}).   
On a technical note, \BI$_{0}$ is essentially \BI$_{\textup{qf}}$ from \cite[p.\ 78]{troelstra1} for $P(n)\equiv Q(n)$ quantifier-free.  
We now prove $\BI_{0}^{\st}$ form \ref{STP}.  
\begin{theorem} \label{ebimaki2}
In $\P+\ref{STP}$, we have $\BI_{0}^{\st}$.   
\end{theorem}
\begin{proof}
Assume \eqref{cond2b}$^{\st}$ and suppose we have $\neg Q(\langle\rangle)$.  
Now define $F(x,M):=(\mu m\leq M)\neg Q(x*\langle m\rangle)$ and put $G(0):=F(\langle\rangle, M)$ and $G(n+1):=F(G(0)*\dots *G(n), M)$. 
By \eqref{cond2b}$^{\st}$ for $t=\langle\rangle$ and the assumption $\neg Q(\langle\rangle)$, $G(0)$ is standard.  
Furthermore, we also have that $G(n+1)$ is standard if $G(k)$ is standard, for standard $n$ and $k\leq n$, by \eqref{cond2b}$^{\st}$ for $t=G(0)*\dots*G(n)$. Hence, $G(n)$ is standard for all standard $n$ by external induction \textsf{IA}$^{\st}$.
This in turn implies that $\neg Q(G(0)*\dots*G(k))$ for standard $k$, by quantifier-free induction and \eqref{cond2b}$^{\st}$.
Now consider the sequence $\beta^{1}=G(0)*G(1)*G(2)*\dots$ and let $\gamma^{1}$ be its standard part via $\STP$.  Finally, apply \eqref{cond2b}$^{\st}$ for $\alpha=\gamma$ to obtain a contradiction.  
\end{proof}
The theorem is not that surprising: \ref{STP} is the nonstandard version of $\WKL$ by Theorem~\ref{agda}, the latter lemma is equivalent to a version of dependent choice (See \cite[VIII.2.5]{simpson2}), and bar induction is a version of the latter.  

\medskip

Nonstandard Analysis also has a `distinct' kind of induction, called \emph{external induction}. as follows:
\begin{princ}[\textsf{ExInd}]
For standard $F^{(0\times0)\di 0}$ and $M\in \Omega$, if
\be\label{EI}
\st(F(0,M))\wedge (\forall^{\st}n)[\st(F(n,M))\di \st(F(n+1,M))],
\ee
then $(\forall^{\st}n)(\st(F(n,M))$.  
\end{princ}
Intuitively speaking, (\textsf{ExInd}) tells us that we may use induction on the new standardness predicate \emph{along the standard numbers} (and obviously not along all the numbers).      
Although seemingly more general than normal induction, we now derive \textsf{(ExInd)} from the standardness of the recursor constants in $\P$.  We consider this theorem as its proof is similar to the proof of Theorem \ref{ebimaki}.  
\begin{theorem}\label{reform} 
The system $\P\setminus\{ \textsf{\textup{IA}}^{\st}\}$ proves \textup{(\textsf{ExInd})}.  
\end{theorem}
\begin{proof}
Consider \eqref{EI} and replace `\st' as follows:
\[
 (\forall^{\st}n)[(\exists^{\st}k^{0})(F(n,M)\leq k)\di (\exists^{\st}l^{0})(F(n+1,M)\leq l)\big].   
\]
Now bring all standard quantifiers outside to obtain:
\be\label{handje}
 (\forall^{\st}n, k)(\exists^{\st}l)[F(n,M)\leq k\di F(n+1,M)\leq l].
\ee
Recall Remark \ref{simply} and apply $\HAC_{\INT}$ to \eqref{handje} to obtain standard $g^{1}$ such that %
\[
(\forall^{\st}n, k)[F(n,M)\leq k\di F(n+1,M)\leq g(n,k)].
\]
Now use primitive recursion to define the \emph{standard} function $h^{1}$ such that $h(0):=F(0,M)$ and 
$h(k+1):=g(k, h(k))$.  By the definition of $h$, we have $F(n, M)\leq h(n)$ for standard $n$, proved by quantifier-free induction (of the non-external variety).  As $h(n)$ is standard for standard $n$, \eqref{EI} implies the consequent of \textsf{(ExInd)}.    
\end{proof}
Note that the same proof goes through for variations of \textup{(\textsf{ExInd})}, e.g.\ if the induction hypothesis involves $(\forall k\leq n)(\st(F(k,M)))$ instead of $\st(F(n,M))$.  
Note that \textsf{(ExInd)} also follows directly from \textsf{IA}$^{\st}$, but the latter cannot be included in fragments of $\P$ based on \textsf{E-PRA}$^{\omega}$ (See \cite{kohlenbach2}*{\S2}).  

\medskip

We now formulate \emph{external bar induction}, which is bar induction on the (external) standardness predicate.  
\begin{princ}[\textsf{EBI}]\label{EBI}
For standard $F^{(0\times 0)\di 0}$ and $M\in \Omega$, if
\begin{align}
&(\forall^{\st} \alpha^{1} )(\exists^{\st} n^{0})\big[\st(F(\overline{\alpha}n,M))  \big] \label{cond1} \\
\wedge & (\forall^{\st} t^{0})\big[(\forall^{\st} x^{0})(\st(F(t*\langle x\rangle ,M)))\di \st(F(t,M))\big] \label{cond2}
\end{align}
then $\st(F(\langle\rangle,M))$.  
\end{princ}
Finally, we prove external bar induction from $\STP$ and the following fragment of Nelson's axiom \emph{Transfer}:  
\be\tag{$\paai$}
(\forall^{\st}f^{1})\big[(\forall^{\st}n^{0})f(n)\ne0\di (\forall m^{0})f(m)\ne0\big].
\ee
\begin{theorem} \label{ebimaki}
The system $\P+\ref{STP}+\paai$ proves \textup{\EBI}.   
\end{theorem}
\begin{proof}
First of all, we use $\paai$ to obtain \emph{standard} $\mu^{2}$ such that $[\MU(\mu)]^{\st}$.  To this end, define $\nu^{2}$ as follows:
\[
\nu(f, N):=
\begin{cases}
(\mu n\leq N)f(n)=0 & (\exists n\leq N)(f(n)=0) \\
0 & \textup{otherwise}
\end{cases}.
\]
Assuming $\paai$, we have $(\forall^{\st}f^{1})(\forall N,M\in \Omega)(\nu(f, M)=_{0}\nu(f, N))$, i.e.\ $\nu(\cdot, N)$ is $\Omega$-invariant.  
Now let $\mu^{2}$ be the standard part of $\nu(\cdot, N)$ provided by $\OCA$, and note that $[\MU(\mu)]^{\st}$, i.e.\ $\mu^{2}$ is Feferman's search operator relative to `st'.  

\medskip

Secondly, consider \eqref{cond2}, 
and bring the latter in the form:
\be\label{kurgf}
(\forall^{\st} t^{0})(\exists^{\st}x^{0},m^{0} )(\forall^{\st}l^{0})\big[F(t*\langle x\rangle ,M)\leq l\di F(t,M)\leq m\big].  
\ee
Let $f^{1}\leq_{1}1$ be the characteristic function of the formula in square brackets in \eqref{kurgf}, and let \emph{standard} $h^{1}\leq_{1}1$ be the standard part of $f^{1}$ as provided by $\STP$.   
Hence, \eqref{kurgf} implies that
\be\label{fasjoe}
 (\forall^{\st} t^{0})(\exists^{\st}x^{0},m^{0} )(\forall^{\st}l^{0})(h(t, x, m,l)=1), 
\ee 
Similar to the proof of Corollary \ref{seconde}, $\mu^{2}$ from $[\MU(\mu)]^{\st}$ 
can be used to define a standard characteristic function $\chi^{1}$ for $(\forall^{\st}l^{0})(h(t, x, m,l)=1)$ from \eqref{fasjoe}, i.e.\ 
\[
(\forall^{\st}x, m, l)\big[(\forall^{\st}l^{0})(h(t, x, m,l)=1) \asa \chi(x, m, l)=1].  
\]
Applying $\mu^{2}$ as in $[\MU(\mu)]^{\st}$ to $ (\forall^{\st} t^{0})(\exists^{\st}x^{0},m^{0} )(\chi(x, m,l)=1)$ 
then yields standard $g^{1}$ such that $ (\forall^{\st} t^{0})(\forall^{\st}l^{0})(h(t, g(t)(1), g(t)(2),l)=1)$, which implies:
\be\label{eqrefke}
(\forall^{\st} t^{0},l^{0})\big[F(t*\langle g(t)(1)\rangle ,M)\leq l\di F(t,M)\leq g(t)(2)\big].    
\ee
Finally, we derive \textsf{EBI} using $g$ from \eqref{eqrefke}.  Thus, assume \eqref{cond1} and \eqref{cond2} and consider the standard $\alpha^{1}$ defined by $\alpha(0):=g(\langle\rangle)(1)$ and $\alpha(n+1):= g(\overline{\alpha}n)(1)$.  
By \eqref{cond1}, there is standard $n$ such that $F(\overline{\alpha}n, M)$ is standard.  However, then there is standard $l$ such that $F(\overline{\alpha}n, M)=F(\overline{\alpha}(n-1)*\langle g(\overline{\alpha}(n-1))(1)\rangle, M)$ satisfies the antecedent of \eqref{eqrefke} for $t=\overline{\alpha}(n-1)$.  
Hence, $F(\overline{\alpha}(n-1), M)$ is also standard by \eqref{eqrefke}, as the latter yields $F(\overline{\alpha}(n-1), M)\leq g(\overline{\alpha}(n-1))(2)$. 
Applying \eqref{eqrefke} for $t=\overline{\alpha}(n-2)$ and $l=g(\overline{\alpha}(n-1))(2)$, we obtain $F(\overline{\alpha}(n-2), M)\leq g(\overline{\alpha}(n-2))(2)$.  
Applying the same procedure at most $n$ times, we obtain $F(\langle\rangle, M)\leq g(\langle\rangle)(2)$, i.e.\ $F(\langle \rangle, M)$ is standard, and \textsf{EBI} follows.    
\end{proof}

\section{The Gandy-Hyland functional in Nonstandard Analysis}\label{drei}
\subsection{Introduction}
In this section, we prove our main results concerning the Gandy-Hyland functional $\Gamma$ from $\eqref{GH}$ and its so-called canonical approximation $G$, defined as follows:
\be\label{small233}
G(Y,s,M)=
\begin{cases}
Y(s*00\dots) & |s|\geq M \\
Y(s*0* (\lambda n)G(Y, s*(n+1),M)) & \textup{otherwise}   
\end{cases}.
\ee
As to its provenance, we recall that the $\Gamma$-functional was introduced in \cite{gandymahat} as an example of a functional not Kleene-S1-S9-computable over the total continuous functionals, even with the fan functional as an oracle (See \cite[\S4]{noortje} or \cite{longmann}*{\S8}). 
By contrast, $G$ is primitive recursive, as discussed in Section \ref{TITI}.  

\medskip

Using the results from the previous section, we prove in Section~\ref{deballen} that the Gandy-Hyland functional $\Gamma(\cdot)$ equals $G(\cdot, M)$ 
for all standard inputs and nonstandard $M$;  This proof takes place in an extension of the system $\P$ from Section \ref{PIPI}.   
From this nonstandard proof, we extract a term from G\"odel's \textsf{T} expressing $\Gamma$ in terms of the special fan functional (See Corollary \ref{scruf}) and 
a modulus-of-continuity functional.   
This final result \emph{does not involve Nonstandard Analysis}.    

\medskip

In Section~\ref{laiba} to \ref{quatre}, we obtain similar nonstandard theorems, from which we extract the associated relative computability results.  
In particular, in Sections~\ref{laiba} and \ref{ballen3}, we prove `pointwise' versions of the above results, not involving a modulus-of-continuity functional.   
We introduce the well-known notion of \emph{associate} of a continuous functional in Section \ref{quatre}, and use it to obtain particularly elegant results.  In our opinion, the aforementioned variations of the main result establish the robustness of our approach.   

\medskip

Finally, we show in Section \ref{relletje} that one can \emph{re-obtain} the original nonstandard theorem 
(that the Gandy-Hyland functional $\Gamma(\cdot)$ equals $G(\cdot, M)$ for all standard inputs and nonstandard $M$) from the proof of a certain natural relative computability result, called the \emph{Herbrandisation} 
of the original nonstandard theorem.  In this way, the latter is seen to have the same computational content as its Herbrandisation.     
Based on the results in Section \ref{relletje}, one can easily obtain the Herbrandisation for \emph{any} nonstandard theorem in this paper.  The observed connection between a nonstandard theorem and its `highly constructive' Herbrandisation, provides us with a two-way street between the fields Nonstandard Analysis and Computability.    

\subsection{From Nonstandard Analysis to relative computability}\label{deballen}
In this section, we prove that the functionals $G(\cdot, M)$ and $\Gamma(\cdot)$ are equal for standard inputs and nonstandard $M^{0}$, inside an extension of $\P$.  
From this proof, we extract a term from G\"odel's system $\textsf{T}$ which computes $\Gamma$ in terms of a modulus-of-continuity functional and the special fan functional (See Corollary \ref{scruf}).   

\medskip

As noted in the first section, the $\Gamma$-functional corresponds to modified bar recursion of type 0 (See \cite[\S4]{bergolijf}).  
Since bar recursion holds in the model of all total continuous functionals (See \cite{ershov, bergolijf2}), the easiest way of obtaining $\Gamma$ from $G$ seems to be adding the 
continuity axiom $\NPC$ to $\P$, which was defined above as:
\be\label{druk2}\tag{\textup{\textsf{NPC}}}
(\forall^{\st}Y^{2}\in C, f^{1})(\forall g^{1})(f\approx_{1}g\di Y(f)=_{0}Y(g)),
\ee
where `$Y^{2}\in C$' is the (internal) definition of continuity as in \eqref{CB2}.  
As discussed in Remark \ref{rukker}, \ref{druk2} without the restriction `$Y^{2}\in C$' is inconsistent, while $\ref{druk2}$ easily\footnote{Fix a standard $Y^{2}$ and standard $f^{1}$ in \eqref{CB2}, and apply (the contraposition of) \emph{Transfer}.} follows from the $\IST$ axiom \emph{Transfer} (See Theorem~\ref{loeder}).       

\medskip

Furthermore, according to \cite[p.\ 167]{bergolijf}, the role of the continuity principle and bar induction in \cite[Theorem 2.5]{bergolijf} is \emph{to verify the correctness of the \emph{[}bar recursive\emph{]} witnessing functional}.  
As was proved in Section \ref{exbisalal}, the principles \ref{STP} and $\sigtoe$ from Section~\ref{monga} yield a version of bar induction and a nonstandard continuity principle $\NPC$.  Hence, we arrive at the following theorem.
\begin{theorem}\label{drag}
In $\P$\textup{ + \ref{predruk} + \ref{STP}}, we have
\be\label{alltime2}
(\forall^{\st}Y^{2}\in C,s^{0})(\forall M,N\in \Omega)(G(Y,s,N)=_{0}G(Y,s,M)), 
\ee
i.e.\ the canonical approximation of $\Gamma$ is $\Omega$-invariant.
\end{theorem}
\begin{proof}
We sketch the proof of the theorem and then provide a detailed version.

\medskip

First of all, \textsf{EBI} and nonstandard continuity $\NPC$ may be used in light of Theorem \ref{loeder} and Theorem \ref{ebimaki}.  
Secondly, one uses this bar induction to prove that $G(Y, s,M)$ is standard for standard $Y^{2}\in C, s^{0}$ and nonstandard $M^{0}$.  Thirdly, one applies bar induction again to prove that \eqref{alltime2} holds for fixed inputs.  In both cases, nonstandard continuity is used to establish \eqref{cond1} and \eqref{cond2} in external bar induction.   
We now provide a detailed proof. 

\medskip

We first prove that $G(\cdot,M)$ is standard for standard input and nonstandard $M$ using \textsf{EBI}.  
To this end, fix standard $Y^{2}\in C,s^{0}$ and $M\in \Omega$, and define $F(x^{0},M):=G(Y,s*x,M)$.  
To prove \eqref{cond1}, fix standard $\gamma^{1}$ and $N\in \Omega$. We have 
\begin{align}
F(\overline{\gamma}N,M) &=G(Y,s*\overline{\gamma}N, M)\notag\\
&=Y(s*\overline{\gamma}N*0*(\lambda n)G(Y,s*\overline{\gamma}N*(n+1),M))\label{adoklol2}\\
&=Y({s*{\gamma}}),\label{adoklol}
\end{align}
where the final step follows by nonstandard continuity $\ref{druk2}$ as $s*\gamma \approx_{1} \zeta$, where the latter is the sequence in \eqref{adoklol2}.  
We have proved that $(\forall K\in \Omega)F(\overline{\gamma}K)=Y(s*\gamma)$ and 
underspill yields $(\exists^{\st}k^{0})(\forall K\geq k)F(\overline{\gamma}K)=Y(s*\gamma)$, from which it is immediate that $(\forall^{\st}\gamma^{1})(\exists^{\st} m)(\forall n\geq m)(\st(F(\overline{\gamma}n,M)))$, and hence \eqref{cond1}.  

\medskip

To prove \eqref{cond2}, assume the antecedent of the latter for standard $t$, and consider
\begin{align}
F(t,M)&=G(Y,s*t, M)\notag\\
&=Y\big(s*t*0* (\lambda n)G(Y, s*t*(n+1),M )) \notag\\ 
&=Y\big(s*t*0* (\lambda n)F(t*(n+1),M ))\label{slock}  
\end{align}
which follows by the definitions of $F$ and $G$.  
However, the antecedent of \eqref{cond2} tells us that $F(t*\langle m\rangle,M)$ is standard for standard $m$.  
Hence, the sequence 
\be\label{fuzz}
s*t*0* F(t*1,M )*F(t*2,M )*F(t*3,M )*\dots
\ee
has a standard part by \ref{STP}, say $\gamma^{1}$, and \ref{druk2} yields $F(t,M)=Y(\gamma)$, which is standard.
Hence, we obtain \eqref{cond2}, and $F(\langle\rangle, M)=G(Y,s,M)$ is standard by $\textsf{EBI}$, for any standard $Y^{2}\in C$ and standard $s^{0}$.

\medskip

Secondly, we prove \eqref{alltime2} using the previous part of the proof and \textsf{EBI}. 
Thus, define the function $F(x,M)$ as:
\be\label{fucker}
F(x,M):=
\begin{cases}
0 & G(Y,s*x,M)=G(Y,s*x,M+1) \\
M & \textup{otherwise}
\end{cases},
\ee
where $Y^{2}\in C$ and $s^{0}$ are standard again.  
Repeating the steps from the previous paragraph of the proof, we note that $F(\cdot,M)$ satisfies \eqref{cond1} and \eqref{cond2} for any $M\in \Omega$.  
Hence, \textsf{EBI} yields that $F(\langle\rangle,M)$ is standard for any nonstandard $M$;  As a consequence, we have $G(Y,s,M)=G(Y,s,M+1)$ by definition, for any nonstandard $M$.  
Hence, \eqref{alltime2} is proved, and we are done.  
\end{proof}
\begin{rem}[The essential use of $\textsf{EBI}$]\label{frakkklll}\rm
On a side-note, it seems that $\textsf{EBI}$ is essential for the first part of the above proof, but not for the second part: \eqref{alltime2} follows from applying $\BI_{0}^{\st}$ for $Q(x)\equiv [G(Y,s*x,M)=G(Y,s*x,M+1)]$, \emph{assuming that} 
$G(Y, s, M)$ is standard for standard $Y^{2}\in C, s^{0}$ and nonstandard $M$, as was proved in the first part of the above proof \emph{using} \textsf{EBI} \emph{however}.  
\end{rem}
The Gandy-Hyland is unique as noted in \cite{gandymahat}*{\S6} and \cite{longmann}*{\S8.3.3}.  We prove a similar result, for which we require:  
\be\tag{$\GH_{\st}(\Gamma)$}
(\forall^{\st} Y^{2}\in C, s^{0})\big[\Gamma(Y^{2},s^{0})= Y\big(s*0* (\lambda n^{0})\Gamma(Y, s*(n+1))\big)\big].
\ee
The corollary expresses that the standard and unique Gandy-Hyland functional equals its canonical approximation.  
\begin{cor}\label{Cruxxxx}
In $\P+\ref{predruk}+\ref{STP}$, the Gandy-Hyland functional exists and equals its canonical approximation, i.e.\ there is standard $\Gamma^{3}$ such that $\GH_{\st}(\Gamma)$ and
\be\label{crubm}\tag{\textsf{\textup{CA}}$(\Gamma)$}
(\forall^{\st}Y^{2}\in C,s^{0})(\forall N\in \Omega)(G(Y,s,N)=\Gamma(Y,s)).
\ee
Furthermore, the Gandy-Hyland functional is unique, i.e.\ $(\forall \Gamma_{1}^{3})(\GH_{\st}(\Gamma_{1})\di \textsf{\textup{CA}}(\Gamma_{1}))$.
\end{cor}
\begin{proof}
By \eqref{alltime2}, $G(Y,s,M)$ is $\Omega$-invariant, and $\OCA$ and Remark \ref{lop} yield the standard part of $G(Y,s,M)$, say $\Gamma_{0}(Y,s)$.  
For standard $Y^{2}\in C,s^{0}$ and $M\in \Omega$, we have that:
\begin{align}
\Gamma_{0}(Y,s)
=G(Y,s,M) 
&=Y\big(s*0* G(Y, s*1,M )*G(Y, s*2,M )*\dots \big)\label{durkall}\\
&=Y\big(s*0* \Gamma_{0}(Y, s*1 )*\Gamma_{0}(Y, s*2 )*\dots \big),\notag
\end{align}
where we used \ref{druk2} in the final step.  
Hence, the standard part $\Gamma_{0}(\cdot)$ of $G(\cdot,M)$ as provided by $\OCA$ is indeed the Gandy-Hyland functional as $\GH_{\st}(\Gamma_{0})$ follows from \eqref{durkall}.  
To prove the uniqueness as in the corollary, suppose there is another $\Gamma_{1}$ such that $\textsf{GH}_{\st}(\Gamma_{1})$ and define $F(x, M)$ as in \eqref{fucker}, but with $G(Y,s*x,M+1)$ replaced by $\Gamma_{1}(Y, s*x)$.  Now proceed as in the proof of the theorem to establish that this modified version of \eqref{fucker} satisfies \eqref{cond1} and \eqref{cond2}.  From \textsf{EBI}, we obtain that $\st(F(\langle\rangle, M))$, implying that $G(Y,s,M) =\Gamma_{1}(Y, s)$.  The latter thus holds for standard $Y^{2}\in C, s^{0}$ and $M\in \Omega$, and $\CA(\Gamma_{1})$ follows.  
\end{proof}
As noted above, the Gandy-Hyland functional is not computable (in the sense of Kleene's S1-S9) in terms of the fan functional over the total continuous functionals.  
The following corollaries express that the Gandy-Hyland functional may be computed via a term in G\"odel's $\textsf{T}$ from a modulus-of-continuity functional.  We require the following:    
\be\tag{$\GH(\Gamma)$}
(\forall Y^{2}\in C, s^{0})\big[\Gamma(Y^{2},s^{0})= Y\big(s*0* (\lambda n^{0})\Gamma(Y, s*(n+1))\big)\big].
\ee
Variations of the following relative computability result are discussed below.  
\begin{cor}[Term Extraction I]\label{TE}
From the proof in $\P$ of 
\be\label{nsaversion}
\ref{predruk}+\ref{STP}\di (\forall \Gamma^{3})\big[\textsf{\textup{GH}}_{\st}(\Gamma)\di \textup{\textsf{CA}}(\Gamma)   \big], 
\ee
a term $t^{4}$ can be extracted such that \textsf{\textup{E-PA}}$^{\omega*}+\QFAC^{1,0}$ proves that
\begin{align}\label{consversion}
(\forall \mu^{2} , \Theta^{3},\Gamma^{3} )\big[\big(\textsf{\textup{GH}}(\Gamma)  \wedge&\textup{\textsf{MU}}(\mu)\wedge \SCF(\Theta)\big) \di (\forall Y^{2} \in C, s^{0})\big(G(Y, s, t(Y,s, \mu,\Theta))=\Gamma(Y, s) \big)\big],
\end{align}
i.e.\ $G(Y, s, t(Y, s, \mu,\Theta ))$ is the Gandy-Hyland functional expressed in terms of Feferman's search operator and $\Theta$.  
\end{cor}
\begin{proof}
The following formula is provable in $\P$ by Corollary \ref{Cruxxxx}:
\be\label{fornikal}
(\forall \Gamma^{3})\big[[\textsf{\textup{GH}}_{\st}(\Gamma)\wedge \STP\wedge \ref{predruk}]\di \textup{\textsf{CA}}(\Gamma)   \big].  
\ee
We apply Remark \ref{firliborn}:  Bring all the components of \eqref{fornikal} in normal form.  
Note that a normal form of $\sigtoe$ is given as \eqref{normaltrank}; Let $A(f, n)$ be the internal formula in square brackets in the latter.   

\medskip

Secondly, let $(\forall^{\st}g^{2})(\exists^{\st} w^{1^{*}})B(g, w)$ be the normal form \eqref{teringzooi} of $\STP$  formulated in the proof of Corollary \ref{scruf}.  
Thirdly, Corollary~\ref{Cruxxxx} combined with underspill implies that for all $\Gamma$ such that $\textsf{\textup{GH}}_{\st}(\Gamma)$ we have 
\be\label{tokie}
(\forall^{\st}Y^{2}\in C,s^{0})(\exists^{\st} K)\big[(\forall N\geq K)(G(Y,s,N)=_{0}\Gamma(Y,s))\big],   
\ee
and let $C(Y, s, K, \Gamma)$ be the formula in square brackets in \eqref{tokie}.  Then \eqref{fornikal} implies that for all $\Gamma$, we have
\begin{align}\label{americanland0}
 \big[(\forall^{\st}f^{1})(\exists^{\st}N^{0})A( f, N)\wedge (\forall^{\st}W^{2})(\exists^{\st}w^{1^{*}})B(W, w) \wedge \textsf{\textup{GH}}_{\st}(\Gamma) \big]\di (\forall^{\st}Z^{2}\in C, s)(\exists^{\st}N)C(Z, s, N, \Gamma). 
\end{align}
Hence, for all $\Gamma$ and all \emph{standard} $\Theta, \Psi$, we have
\begin{align}\label{americanland}
 \big[(\forall^{\st}f^{1})A(f, \Psi(f))\wedge (\forall^{\st}W^{2})B(W, \Theta(W)) \wedge \textsf{\textup{GH}}_{\st}(\Gamma) \big]
 \di (\forall^{\st}Z^{2}\in C, s)(\exists^{\st}N)C(Z, s, N, \Gamma), 
\end{align}
as standard functionals have standard outputs for standard inputs.  
Strengthening the antecedent of \eqref{americanland} to internal formulas, we have for $\Gamma$ and \emph{standard} $\Theta, \Psi$ that
\begin{align}\label{americanland1227}
 \big[(\forall  f^{1})A(f, \Psi(f))\wedge (\forall W^{2})B(W, \Theta(W)) \wedge \textsf{\textup{GH}}(\Gamma) \big]
 \di (\forall^{\st}Z^{2}\in C, s)(\exists^{\st}N)C(Z, s, N, \Gamma). 
\end{align}
Bringing outside all standard quantifiers, \eqref{americanland1227} implies:
\begin{align}
(\forall^{\st}\Theta, \Psi, Z\in C, s)(\forall \Gamma)(\exists^{\st}N^{0})\label{americanland2}
\big[ \big((\forall f)A(f, \Psi(f))\wedge (\forall W^{2})B(W, \Theta(W))\wedge \textup{\textsf{GH}}(\Gamma) \big)\di C(Z, s, N, \Gamma)   \big], 
\end{align}
where $D(\cdots)$ is the internal formula in square brackets.  
Recall Remark~\ref{simply} and apply idealisation \textsf{I} to \eqref{americanland2}:
\begin{align}\label{liko}
&(\forall^{\st}\Theta, \Psi, Z\in C, s)(\exists^{\st}N^{0})(\forall \Gamma)D(\Theta, \Psi, Z, s, N, \Gamma).
\end{align}
Apply Corollary \ref{consresultcor} to `$\P\vdash \eqref{liko}$' to obtain a term $u$ such that $\textsf{E-PA}^{\omega*}+\QFAC^{1,0}$ proves
\[
(\forall \Theta, \Psi, Z\in C, s)(\exists N^{0}\in u(\Theta, \Psi, Z, s))(\forall \Gamma)D(\Theta, \Psi, Z, s, N, \Gamma), 
\]
and define $t(\Theta, \Psi, Z, s):=\max_{i<|u(\Theta, \Psi, Z, s)|}u(\Theta, \Psi, Z, s)(i)$.  Now note that 
\be\label{kuhhh}
(\forall \Theta, \Psi, Z\in C, s, \Gamma)D(\Theta, \Psi, Z, s, t(\Theta, \Psi, Z, s), \Gamma), 
\ee
due to the monotone behaviour of $C(Z, s, \cdot, \Gamma))$.
Bringing the $Z$ and $s$ quantifiers into the consequent of $D$ in \eqref{kuhhh}, we obtain that for all $\Theta, \Psi, \Gamma$:
\[
 [(\forall f)A(f, \Psi(f))\wedge \SCF(\Theta)\wedge \textup{\textsf{GH}}(\Gamma) ]\di (\forall Z^{2}\in C,s^{0})C(Z, s, t(\Theta, \Psi, Z, s), \Gamma).    
\]
Finally, we note that $\Psi$ as in $(\forall f)A(f, \Psi(f))$ is Feferman's second search operator as in $\MU_{2}(\nu)$.  In the proof of Theorem~\ref{seconde}, the latter functional is explicitly defined in terms of Feferman's search operator as in $\MU(\mu)$.    
%
%
%
%
%
\end{proof}
Recall Remark \ref{simply} and note that we performed a similar procedure as in the former remark (involving a term of type $0^{*}$ and its maximum) to obtain \eqref{kuhhh} from \eqref{liko} after applying Corollary \ref{consresultcor} to the latter.  Hereon-after, we will sometimes skip the (obvious) step involving the maximum when applying Corollary \ref{consresultcor} too.   
\begin{cor}\label{forlk}
We can obtain a version of \eqref{consversion} with Feferman's search operator $\MU(\mu)$ replaced by $\MPC(\Psi)$, i.e.\ the Gandy-Hyland functional can be expressed 
in terms of a modulus-of-continuity functional and $\Theta$.  
\end{cor}
\begin{proof}
From $\Psi$ as in $\MPC(\Psi)$, one can define a \emph{discontinuous} type two functional (See \cite{exu} and \cite{beeson1}*{Theorem 19.1}). 
By \cite{kohlenbach2}*{Prop.\ 3.7} and \cite{kohlenbach3}*{\S3}, a discontinuous type two functional can be used to define $\mu^{2}$ as in $\MU(\mu)$, 
using choice functionals originating from the application of $\QFAC^{1,0}$.    
\end{proof}
We now discuss a possible strengthening of the above results.  
\begin{rem}[Similar results]\label{rukker}\rm
It is an interesting question if the condition `$Y^{2}\in C$' in $\NPC$ can be weakened.  First of all, we \emph{cannot} drop this condition:  As shown in the proof of Corollary \ref{seconde}, \ref{druk2} gives rise to a modulus-of-continuity functional $\Psi$ as in \ref{lukl2}.  
From the latter functional, one constructs a \emph{discontinuous} type two functional (See \cite{exu} and \cite{beeson1}*{Theorem 19.1}), which contradicts $\MPC(\Psi)$ without the restriction `$Y^{2}\in C$'.  Secondly, going through the proofs in this section, it seems that `$Y^{2}\in C$' can be replaced by any \emph{internal} formula $D(Y^{2})$, as long as the latter formula blocks the aforementioned contradiction in the same way as `$Y^{2}\in C$' does.
\end{rem}
In Section \ref{laiba}, we obtain relative computability results similar to \eqref{consversion} with weaker antecedents.  
We now discuss how the \emph{consequent} of the above results can be strengthened.  
\begin{rem}[Similar results II]\rm
It is a natural question if the consequent of \eqref{consversion} is the best possible.  A careful study of the proofs of Theorem \ref{drag} and Corollary \ref{Cruxxxx} reveals the existence of standard $\Gamma^{3}$ such that $\GH_{\st}(\Gamma)$ and
\begin{align}\label{groin}
(\forall^{\st}Y^{2}\in C, \alpha^{1})(\forall N^{0}, M^{0}\in \Omega)(\forall s^{0})(\overline{\alpha}M=_{0}\overline{s*00\dots}M\di \Gamma(Y, s)=_{0}G(Y, s, N) ).
\end{align}
Indeed, if $s^{0}$ is standard, then the associated instance of \eqref{groin} follows from Corollary \ref{Cruxxxx}.  If $s^{0}$ is nonstandard, it has $\alpha$ as a standard part and \eqref{groin} follows by nonstandard continuity as in $\NPC$.  Applying term extraction to a variation of \eqref{nsaversion} involving \eqref{groin}, one obtains a term $u$ which computes the Gandy-Hyland functional `more uniformly' than the term $t$ in \eqref{consversion}, in that $u$ provides \emph{one} stopping condition for every initial segment $s^{0}$ of the sequence $\alpha^{1}$.  
We shall derive $\NPC$ from \eqref{groin} in Section \ref{ballen3}.  
\end{rem}
\begin{rem}[Similar results III]\label{flir}\rm
We now discuss whether the previous results go through inside a fragment of $\P$.
The `good' news is that our term extraction results, namely Theorem \ref{consresult} and Corollary \ref{consresultcor}, do not really depend on the presence of full Peano arithmetic.  In particular, it is an easy verification that the proof of \cite{brie}*{Theorem 7.7} goes through for any fragment of \textsf{E-PA}$^{\omega{*}}$ which includes \textsf{EFA}, sometimes also called $\textsf{I}\Delta_{0}+\textsf{EXP}$.
The `bad' news is that in light of \cite{escaleert}*{\S5}, it seems that in order to define the canonical approximations $G$, one cannot avoid invoking a principle (slightly) stronger than primitive recursive arithmetic (\cite{buss}*{\S1.2.10}).
\end{rem}
\begin{rem}[Similar results IV]\label{flir2}\rm
We now discuss whether the use of $\NPC$ in Theorem \ref{drag} is necessary.  On one hand, it seems one can replace the use of $\NPC$ in the proof of the latter by: nonstandard uniform continuity as in
\be\label{wkb}
(\forall^{\st} Y^{2}\in C, h^{1})(\forall f^{1}, g^{1}\leq_{1}h )({f}\approx_{1}{g}\di Y(f)=_{0}Y(g)).   
\ee
and nonstandard `weak' continuity as follows: 
\be\label{wka}
(\forall^{\st}f^{1}, Y^{2}\in C)(\forall g^{1})(f\approx_{1}g\di \st(Y(g))).  
\ee
Applying the template from Remark \ref{firliborn}, \eqref{wkb} and \eqref{wka} give rise to the \emph{fan functional} and the \emph{weak continuity functional} (See \cite{bergolijf} for the latter).  
On the other hand, the proof of Theorem \ref{ebimaki} suggests that external bar induction \textsf{EBI} requires \emph{Transfer}, and the latter gives rise to Feferman's search operator.  
Hence, it seems adopting \eqref{wkb} and \eqref{wka} does not yield a version of \eqref{consversion} not involving Feferman's search operator.  
\end{rem}

In conclusion, we have proved in Theorem \ref{drag} that the functional $G(\cdot, M)$ and $\Gamma(\cdot)$ are equal for standard inputs and nonstandard $M^{0}$.  
From this proof, we have extracted a term from G\"odel's \textsf{T} which computes the $\Gamma$-functional as a function of a modulus-of-continuity functional.   
While these relative computability results are not necessarily deep or surprising, our methodology constitutes the true surprise:  That from the proof of Theorem~\ref{drag}, in which no attention to effective content is given, \emph{and involving Nonstandard Analysis}, the term $t$ as in Corollary \ref{TE} may be extracted.   
We prove variations of these results in Sections \ref{laiba} to \ref{quatre}, establishing the robustness of our approach.  

\subsection{From relative computability to Nonstandard Analysis}\label{relletje}
In the previous section, we showed how to extract relative computability results like \eqref{consversion} from corresponding nonstandard statements like \eqref{nsaversion}.  
Now, it is a natural `Reverse Mathematics style' question whether it is possible to re-obtain the nonstandard implication
from (a variation of) the associated relative computability result.  

\medskip

Another natural question is whether we can obtain a version of \eqref{consversion} with weaker assumptions;  Indeed, to compute $\Gamma(Y, s)$ it should -intuitively speaking- suffice to have a functional which (only) behaves like the special fan and modulus-of-continuity functional for $Y$ (and functionals explicitly defined from the latter).   

\medskip

To answer these two questions, we define the \emph{Hebrandisation} of \eqref{consversion} as follows.
Let $\SCF(\Theta, g)$ be $\SCF(\Theta)$ with the leading quantifier involving $g$ dropped.  
Let $\MU_{2}(\nu, f)$ be $\textsf{MU}_{2}(\nu )$ from Section \ref{seqtor} with the leading quantifier involving $f$ dropped.  
Let \textsf{GH}$(\Gamma, Y, s)$ be $\textsf{GH}(\Gamma )$ with the quantifier involving $Y$ and $s$ dropped.  
\begin{definition}[Herbandisation]\label{HERWEG}\rm
Let $i^{3\di 2^{*}}$ and $o^{4}$ be terms from the language of $\textsf{E-PA}^{\omega*}$. 
The \emph{Herbrandisation} \textsf{HER}$(i,o)$ of \eqref{fornikal} is the statement that for all $\Xi=(\Theta, \nu)$ and all $\Gamma^{3},Y^{2}\in C, s^{0}$
\begin{align}
 \big[\big(\forall (Z^{2}\in C, t^{0})\in i(Y, s, \Xi)(1)\big))\textsf{\textup{GH}}(\Gamma, Z, t) \wedge(\forall W^{2}\in i(Y, s, \Xi)(2))&\SCF(\Theta, W)\notag \wedge \big(\forall f^{1} \in i(Y, s, \Xi)(3)\big)\textup{\textsf{MU}}_{2}(\nu, f)\big]\\
& \di (\forall M\geq  o(Y,s, \Xi))\big(G(Y, s,M)=\Gamma(Y, s) \big).\notag 
\end{align} 
\end{definition}
Intuitively speaking, \textsf{HER}$({i,o})$ expresses that to approximate $\Gamma(Y, s)$ via its canonical approximation $G$ involving the term $o$, it suffices that $\Theta$ and Feferman's second search operator satisfy their usual definition \emph{on the restriction of their domains provided by $i$}.  
By the following theorem, the nonstandard version \eqref{fornikal} is `meta-equivalent' to its Herbrandisation in that a proof of the former can be converted into a proof of the latter, and vice versa  
\begin{theorem}\label{TEH}
From the proof of \eqref{fornikal} in $\P$, two terms $i, o$ can be extracted such that \textsf{\textup{E-PA}}$^{\omega*}+\QFAC^{1,0}$ proves $\textsf{\textup{HER}}(i, o)$.  
Moreover, if there are terms $i, o$ such that \textsf{\textup{E-PA}}$^{\omega*}+\QFAC^{1,0}$ proves $\textsf{\textup{HER}}(i, o)$, then $\P$ proves \eqref{fornikal}.  
\end{theorem}
\begin{proof}
For the first part of the theorem, consider the proof of Corollary~\ref{TE} and note that \eqref{americanland} implies (with the same notations as in the aforementioned proof):
\begin{align}
(\forall^{\st}\Theta, \Psi, Z\in C, s)(\forall \Gamma^{3})(\exists^{\st}N^{0},  f^{1}, g^{2}, V^{2}\in C, t^{0})\label{dilllo}
\big[ [A(f, \Psi(f))\wedge \SCF(\Theta, g)\wedge \textup{\textsf{GH}}(\Gamma, V, t) ]  \di C(Z, s, N, \Gamma)   \big],
\end{align}
by pushing outside, as far as possible, the standard quantifiers in \eqref{americanland}.  
Now apply idealisation \textsf{I} to \eqref{dilllo} to obtain:
\begin{align*}
(\forall^{\st}\Theta, \Psi, Z\in C, s)(\exists^{\st}W)(\forall \Gamma^{3})(\exists (N, f, g, V\in C, t)\in W)
\big[ [A( f, \Psi(f))\wedge \SCF(\Theta,g)\wedge \textup{\textsf{GH}}(\Gamma, V, t) ]  \di C(Z, s, N, \Gamma)   \big],
\end{align*}
and apply Corollary \ref{consresultcor} to obtain a term $w$ such that $\textsf{E-PA}^{\omega*}+\QFAC^{1,0}$ proves for all $ \Theta, \Psi, Z\in C, s$ that  
\begin{align*}
&(\exists W\in w(\Theta, \Psi, Z, s))(\forall \Gamma)(\exists (N, f, g, V\in C, t)\in W)
\big[ [A( f, \Psi(f))\wedge \SCF(\Theta, g)\wedge \textup{\textsf{GH}}(\Gamma, V, t) ]  \di C(Z, s, N, \Gamma)   \big].
\end{align*}
Now define the term $o$ as follows: $o(\Theta, \Psi, Z, s)$ is the maximum of the components of $w(\Theta , \Psi, Z, s )$ pertaining to $N$;  
Similarly, define the terms $i(\Theta, \Psi, Z, s)(j)$ for $j=1$ (resp.\ $j=2$ and $j=3$) to be the finite sequence of all components of $w$ pertaining to the variable $  f$ (resp.\ the variable $g$ and the variables $V, t$).  
With these notations, the previous implies for functionals $\Xi=(\Theta, \Psi)$ and $Z^{2}\in C, s^{0}$ that:
\begin{align}
 \big[\big(\forall (V\in C, t^{0})\in i(Z, s, \Xi)(3)\big))\textsf{\textup{GH}}(\Gamma, V, t) \wedge(\forall g^{2}\in i(Z, &s, \Xi)(2))\SCF(\Theta, g)
 \wedge \big(\forall f^{1} \in i(Z, s, \Xi)(1)\big)A( f, \Psi(f))\big]\notag\\
 &\di (\forall M\geq  o(Z,s, \Xi))\big(G(Z, s,M)=\Gamma(Z, s) \big).\label{firker} 
\end{align}
Note that \eqref{firker} is $\textsf{HER}(i,o)$ with slightly different notations.  

\medskip

For the second part, if there are terms $i, o$ such that \textsf{\textup{E-PA}}$^{\omega*}\vdash\textsf{\textup{HER}}(i, o)$, then $\P\vdash\big[\textsf{\textup{HER}}(i, o)\wedge\st(i)\wedge \st(o)\big]$ by the second standardness axiom from Definition~\ref{debs}.  Thus, for \emph{standard} $\Xi=(\Theta, \nu)$ and standard $Y^{2}\in C, s^{0}$, the terms $i(Y, s, \Xi)$ and $o(Y, s, \Xi)$ are standard by the third standardness axiom from Definition~\ref{debs}, and \textsf{HER}$(i,o)$ implies the following weakening (for any $\Gamma^{3}$ and standard $\Theta, \nu, Y^{2}\in C, s$):   
\begin{align}
 \big[(\forall^{\st} Z^{2}\in C, t^{0})\textsf{\textup{GH}}(\Gamma, Z, t) \wedge(\forall^{\st} W^{2})\SCF(\Theta, W)\label{mnklten}
 \wedge (\forall^{\st}  f^{1})\textup{\textsf{MU}}_{2}(\nu, f)\big]\di (\forall M\in \Omega)\big(G(Y, s,M)=\Gamma(Y, s) \big).
\end{align}
Applying $\HAC_{\INT}$ to \eqref{normaltrank}, $\sigtoe$ yields $(\exists^{\st}\nu)(\forall^{\st}  f^{1})\textup{\textsf{MU}}_{2}(\nu, f)$.  
Similarly, $\STP$ implies $(\exists^{\st}\Theta^{3})(\forall^{\st}Y^{2})\SCF(\Theta, Y)$ by the proof of Corollary \ref{scruf}.   
Hence, $\ref{druk2}+\sigtoe+\STP$ implies the second and third conjunct of the antecedent of \eqref{mnklten}, which yields \eqref{fornikal}, and we are done.    
\end{proof}
Thus, we proved that from a proof of \eqref{fornikal}, terms $i, o$ from G\"odel's \textsf{T} can be extracted satisfying the \emph{Herbrandisation} of \eqref{fornikal}, i.e.\ $o$ computes $\Gamma(Y,s)$ as a function of \emph{approximations enforced by $i$} of the special fan functional and Feferman's second search functional.   
Furthermore, the nonstandard version \eqref{fornikal} in turn follows from its Herbrandisation, i.e.\ the latter are `equivalent in the meta-theory' in the sense of the previous theorem.  

\medskip

Obviously, the Herbrandisation $\HER(i,o)$ of \eqref{fornikal} is much more complicated than \eqref{consversion}.  
This seems to be due to the fact that Feferman's second search operator and the special fan functional can be defined explicitly in terms of Feferman's search operator, 
while the same seems impossible for the restrictions of the latter imposed by the term $i$ in $\HER(i,o)$.  Intuitively speaking, one needs to apply Feferman's search operator 
`infinitely many times' to obtain Feferman's second search operator.  

\medskip

In conclusion, the correspondence exhibited in Theorem \ref{TEH} establishes a direct two-way connection between the field Computability (in particular theoretical computer science) and the field Nonstandard Analysis.      
Indeed, while the relative computability result $\textsf{HER}(i,o)$ could arguably still be passed off as (theoretical) computer science, experience bears out that the nonstandard version \eqref{nsaversion} does not count as such among computer scientists.  We could obtain the (meta-equivalent) Herbrandisation for every nonstandard theorem proved in this paper, but we will not do so in the next sections due to space constraints.

\subsection{From Nonstandard Analysis to relative computability II}\label{laiba}
In this section, we obtain a relative computability result for the Gandy-Hyland functional, \emph{not involving} Feferman's search operator.  
To this end, we shall establish that the proofs of Theorem \ref{drag} 
and Corollary~\ref{Cruxxxx} also go through `in a pointwise fashion', to be understood in the sense of Theorem~\ref{ragnor}.  

\medskip

Recall $\GH(\Gamma, Y, s)$ defined in Section \ref{relletje}, and define $\GH(\Gamma, Y)$ as $(\forall s^{0})\GH(\Gamma, Y, s)$.  
Let $\NSC(Y)$ be:  
\be\tag{$\textup{\textsf{NPC}}(Y)$}
(\forall^{\st}f^{1})(\forall g^{1})(f\approx_{1}g \di Y(f)=_{0}Y(g)), 
\ee
i.e.\ \ref{druk2} with the `$(\forall^{\st} Y^{2}\in C)$' dropped, and let $\ST(\Gamma, Y)$ be $(\forall^{\st}s^{0})(\st(\Gamma(Y,s)))$.   
\begin{theorem}\label{ragnor}
The system $\P+\STP$ proves that for all $\Gamma^{3}$ and $Y^{2}$
\be\label{drfre}
\big[\NSC(Y) \wedge \GH(\Gamma, Y) \wedge \textsf{\textup{ST}}(\Gamma, Y)\big]\di (\forall^{\st}s^{0})(\forall N\in \Omega)(\Gamma(Y,s)=G(Y, s, N)). 
\ee
\end{theorem}
\begin{proof}
Fix $\Gamma^{3}$ and $Y^{2},$ as in the antecedent of \eqref{drfre}.  
As \textsf{EBI} is not available, we shall use $\BI_{0}^{\st}$, which \emph{is} available by Theorem \ref{ebimaki2}.  
Thus, fix standard $s^{0}$ and $N\in \Omega$ and consider $Q(x)\equiv[G(Y,s*x,N)=_{0}\Gamma(Y,s*x)]$.

\medskip

To prove the first conjunct of \eqref{cond2b}$^{\st}$, note that for $M\in \Omega$ and standard $\alpha^{1}$, we have $G(Y, \overline{\alpha}M, N)=Y(\alpha)=\Gamma(Y, \overline{\alpha}M)$, by the nonstandard continuity of $Y^{2}$ and the fact that $(\overline{\alpha}M*\gamma)\approx_{1} \alpha$ for any $\gamma^{1}$.  
As in the proof of Theorem \ref{drag}, one obtains the first conjunct \eqref{cond2b}$^{\st}$ using underspill.   

\medskip

To prove the second conjunct of \eqref{cond2b}$^{\st}$, fix standard $t^{0}$ and assume $Q(t*\langle x\rangle)$ for all standard $x^{0}$.  
By definition, we have $G(Y,s*t*\langle x\rangle,N)=\Gamma(Y,s*t*\langle x\rangle)$ for any standard $x^{0}$, and these numbers are standard by $\ST(\Gamma, Y)$.  
By $\STP$, the sequence $s*t*0*(\lambda n)G(Y, s*t*(n+1)),N)$ has a standard part, say $\gamma^{1}$.  Thus:
\[
s*t*0*(\lambda n)G(Y, s*t*(n+1)),N) \approx_{1} \gamma \approx_{1}s*t*0*(\lambda m)\Gamma(Y, s*t*(m+1)), 
\]
and the nonstandard continuity of $Y$ yields:
\begin{align}
G(Y,s*t,N)
&=Y(s*t*0*(\lambda n)G(Y, s*t*(n+1)),N) \label{frihhh}\\
&=Y(\gamma)=Y(s*t*0*(\lambda m)\Gamma(Y, s*t*(m+1)))=\Gamma(Y,s*t),\notag
\end{align}
which implies that $Q(t)$, and the second conjunct of \eqref{cond2b}$^{\st}$ follows.  
Hence, we have proved \eqref{cond2b}$^{\st}$, yielding $Q(\langle\rangle)$ and $G(Y,s,N)=_{0}\Gamma(Y,s)$ as required, and \eqref{drfre} follows.
\end{proof}
We need the following for Corollary \ref{boei}, where $\PCM(Y^{2},Z^{2})$ expresses that $Z$ is a modulus of continuity for $Y$.
\begin{align}\tag{$\PCM(Y, Z)$}
&(\forall f^{1}, g^{1})(\overline{f}Z(f)=_{0}\overline{g}Z(f)\di Y(f)=_{0}Y(g))\\
\tag{$\GHU(\Gamma, Y, H)$}
&(\forall  s^{0})\big[\Gamma(Y, s)=Y(s*0*(\lambda n)\Gamma(Y, s*(n+1)))\leq H(Y, s)\big],
\end{align}
\vspace{-0.75cm}
\begin{cor}[Term Extraction II]\label{boei}
From the proof in Theorem \ref{ragnor}, a term $t$ can be extracted s.t.\ $\textsf{\textup{E-PA}}^{\omega*}+\QFAC^{1,0}$ proves for $\Xi=(H^{1},Z^{2},\Theta^{3})$ and $\Gamma^{3}, Y^{2}$
\[
\big[\PCM(Y,Z)\wedge \SCF(\Theta) \wedge \GHU(\Gamma, Y,H) \big]\di (\forall s)(\forall N\geq t(s,\Xi))(\Gamma(Y,s)=G(Y, s, N)),  
\]
i.e.\ the Gandy-Hyland functional $\Gamma$ at $Y$ can be approximated via a modulus of continuity of $Y$, the special fan functional, and an upper bound for $\Gamma(Y,\cdot)$.  
\end{cor}
\begin{proof}
The template from Remark \ref{firliborn} applies.  We now sketch how one obtains a normal form for all principles in $\STP\di\eqref{drfre}$;    The normal form \eqref{teringzooi} of $\STP$ has been studied in the proof of Corollary \ref{scruf}.  The normal form of $\NSC(Y)$, obtained in the same way as the normal form \eqref{kunt3} of \ref{kunt} in the proof of Theorem \ref{foor}, is     
\[
(\forall^{\st} f^{1})(\exists^{\st}N^{0})\big[(\forall g^{1})(\overline{f}N=_{0}\overline{g}N\di Y(f)=_{0}Y(g))\big],
\]
and applying $\HAC_{\INT}$, one sees how the modulus of continuity of $Y$ comes about.  
Finally, $\ST(\Gamma, Y)$ has the following normal form: $(\forall^{\st}s^{0})(\exists^{\st}n^{0})(\Gamma(Y, s)\leq_{0} n)$, and applying $\HAC_{\INT}$ one observes where the upper bound $H$ comes from.  The normal form for the consequent of \eqref{drfre} is as follows:
\be\label{krifoppp}
 (\forall^{\st}s^{0})(\exists^{\st}n^{0})(\forall N\geq n)(\Gamma(Y,s)=G(Y, s, N)),
\ee
and follows by underspill.  In step \eqref{cruz} from the template in Remark~\ref{firliborn}, idealisation \textsf{I} needs to be applied to pull the `$(\exists^{\st}n)$' quantifier from \eqref{krifoppp} through the quantifier $(\forall \Gamma^{3}, Y^{2})$ from \eqref{drfre}, taking into account Remark \ref{simply}.  
\end{proof}
Following Definition \ref{HERWEG}, it is easy to define the Herbrandisation of \eqref{drfre} and obtain a result similar to Theorem~\ref{TEH}.  
In particular, this Hebrandisation tells us on which part of Baire space the functional $Y$ should be continuous (with modulus $Z$) to guarantee that $\Gamma$ and $G$ coincide at $Y$.  

\medskip

In conclusion, we have obtained a `pointwise' relative computability result for the Gandy-Hyland functional \emph{not involving Feferman's search operator}.  In particular, the term $t$ from Corollary \ref{boei} allows us to compute approximations of the Gandy-Hyland functional in terms of the special fan functional for any functional $Y$ with a modulus of (pointwise) continuity, and a given upper bound on $\Gamma(Y, \cdot)$.  Finally, the statement \emph{every continuous functional on Baire space has a modulus of pointwise continuity}, is rather weak by \cite{kohlenbach4}*{Prop.\ 4.4 and 4.8}.

\subsection{From Nonstandard Analysis to relative computability III}\label{ballen3}
In this section, we show that the approximation of the Gandy-Hyland functional as in \eqref{groin} implies $\NPC$;  We derive the associated relative computability result in which a term from G\"odel's $\textsf{T}$ expresses a modulus-of-continuity functional in terms of an `approximation' functional as in \eqref{groin} for the Gandy-Hyland functional.  We also sketch a `pointwise' result similar to what was established in Section \ref{laiba}.  

\medskip

First of all, consider the following principle and theorem.
\begin{princ}[\textsf{GHS}$_{\ns}$]\label{GHS}
There is $\Gamma^{3}$ such that $\GH_{\st}(\Gamma)$ and
\[
(\forall^{\st}Y^{2}\in C, \alpha^{1})(\forall N^{0}, M^{0}\in \Omega)(\forall s^{0})(\overline{\alpha}M=_{0}\overline{s*00\dots}M\di \Gamma(Y, s)=_{0}G(Y, s, N) ).
\]
\end{princ}
\begin{theorem}\label{copycat}
In $\P$, we have $ \textup{\textsf{GHS}}_{\ns}\di \ref{druk2}$.  
%
\end{theorem}
\begin{proof}
In a nutshell, to obtain \ref{druk2} from \textsf{GHS}$_{\ns}$, one computes for standard $Y^{2}\in C$ the numbers $Y(\alpha)$ and $Y(\beta)$ using \textsf{GHS}$_{\ns}$ and notes that they are identical if $\alpha\approx_{1}\beta$ for standard $\alpha^{1}$ and any $\beta^{1}$.
In more detail, we first apply underspill to the second conjunct of $\textsf{GHS}_{\ns}$, to obtain that
\[
(\forall^{\st}Y^{2}\in C, \alpha^{1})(\exists^{\st} K^{0})(\forall N^{0}, M^{0}\geq K)(\forall s^{0})(\overline{\alpha}M=_{0}\overline{s*00\dots}M\di \Gamma(Y, s)=_{0}G(Y, s, N) ). 
\]
Applying \HAC$_{\textup{\INT}}$ to the previous formula yields a standard functional $\Xi^{3}$ such that 
\begin{align}
(\forall^{\st}Y^{2}\in C, \alpha^{1})(\forall N^{0}, M^{0}\geq \Xi(Y, \alpha))
(\forall s^{0})(\overline{\alpha}M=_{0}\overline{s*00\dots}M\di \Gamma(Y, s)=_{0}G(Y, s, N) ), \label{drink}
\end{align}
keeping in mind Remark \ref{simply}.  Now fix standard $Y^{2}\in C$ and standard $\alpha^{1}$, and any $\beta^{1}$ such that $\alpha\approx_{1}\beta$.  Since $Y^{2}\in C$, there are numbers $N_{0}^{0}, M_{0}^{0}$ such that 
\begin{align}\label{berry}
&(\forall \gamma^{1})(\overline{\gamma}M_{0}=_{0}\overline{\beta}M_{0}\di Y(\gamma)=_{0}Y(\beta)).\\
\label{berry2}
&(\forall \gamma^{1})(\overline{\gamma}N_{0}=_{0}\overline{\alpha}N_{0}\di Y(\gamma)=_{0}Y(\alpha)).
\end{align}
If $M_{0}$ or $N_{0}$ is standard, we have $Y(\alpha)=Y(\beta)$, and we are done.  
In case $M_{0}$ and $N_{0}$ are nonstandard, we have $Y(\overline{\alpha}N_{0}*00\dots)=Y(\alpha)$ and $Y(\overline{\beta}M_{0}*00\dots)=Y(\beta)$ by \eqref{berry} and \eqref{berry2}.  Also, $\Xi(Y, \alpha)$ is standard, yielding that $Y(\overline{\alpha}\Xi(Y, \alpha)*00\dots)=Y(\overline{\beta}\Xi(Y, \alpha)*00\dots)$ by extensionality.  
The following equalities now follow easily:  
\begin{align}
Y(\alpha)=Y(\overline{\alpha}N_{0}*00\dots)=G(Y,\overline{\alpha}N_{0}, N_{0})
&= \Gamma(Y,\overline{\alpha}N_{0})\label{dung}\\
&=G(Y,\overline{\alpha}N_{0}, \Xi(Y,\alpha))\label{dung2}\\
&=Y(\overline{\alpha}\Xi(Y,\alpha)*00\dots)\notag\\
&=Y(\overline{\beta}\Xi(Y,\alpha)*00\dots)\notag\\
&=G(Y,\overline{\beta}M_{0}, \Xi(Y,\alpha))\notag\\
&=\Gamma(Y,\overline{\beta}M_{0}) \notag\\
&=G(Y, \overline{\beta}M_{0}, M_{0})\notag\\
&=Y(\overline{\beta}M_{0}*00\dots)=Y(\beta).\notag
\end{align}
For instance, to obtain the equality between \eqref{dung} and \eqref{dung2}, one applies \eqref{drink} for $s=\overline{\alpha}N_{0}$ and $N=\Xi(Y, \alpha)$ and $M=N_{0}$.  
The remaining equalities are proved similarly, and we are done.    
\end{proof}
Now define \textsf{GHS}$(\Psi, \Gamma)$ as the following formula:
\[
(\forall Y^{2}\in C, \alpha^{1}, N, M\geq \Psi(Y, \alpha), s^{0})\big(\overline{\alpha}M=\overline{s*00\dots}M\di G(Y, s, N)=_{0}\Gamma(Y, s)) \big), 
\]
which expresses that $\Psi$ witnesses the canonical approximation of $\Gamma$ via $G$ in a `more uniform way' than in $\textsf{CA}(\Gamma)$.  
\begin{cor}[Term Extraction III]\label{TvE}
From the proof \textup{`}$\P\vdash \textsf{\textup{GHS}}_{\ns}\di \ref{druk2}$', a term $t$ can be extracted such that \textsf{\textup{E-PA}}$^{\omega*}+\QFAC^{1,0}$ proves for all $ \Gamma^{3}, \Psi^{3}$ that
\begin{align}\label{tinkel}
\big[\GH(\Gamma)\wedge \textsf{\textup{GHS}}(\Psi, \Gamma)  \big] \di \textup{\textsf{MPC}}(t( \Psi)). 
\end{align}
\end{cor}
\begin{proof}
Analogous to the proof of Corollary \ref{TE}, i.e.\ follow the template from Remark~\ref{firliborn}.  
The normal form of $\NPC$, obtained in the same way as the normal form \eqref{kunt3} of \ref{kunt} in the proof of Theorem \ref{foor}, is     
\be\label{yuppiee}
(\forall^{\st} Y^{2}\in C, f^{1})(\exists^{\st}N^{0}_{0})\big[(\forall g^{1})(\overline{f}N=_{0}\overline{g}N\di Y(f)=_{0}Y(g))\big],
\ee
and applying $\HAC_{\INT}$, one sees how the modulus-of-continuity functional comes about.    
In step \eqref{cruz} from the template in Remark~\ref{firliborn}, idealisation \textsf{I} needs to be applied to pull the `$(\exists^{\st}N_{0})$' quantifier from \eqref{yuppiee} through the quantifier $(\forall \Gamma^{3})$ in $\GHS_{\ns}\di\NPC$, taking into account Remark \ref{simply}.  
\end{proof}
%
%
Next, we sketch a `pointwise' version of Theorem \ref{copycat} and Corollary \ref{TvE}, similar to the results in Section \ref{laiba}.
Define the `pointwise' version of $\GHS_{\ns}$ as follows:
\begin{princ}[$\textsf{GHS}_{\ns2}(Y, \Gamma)$]\label{GHS2}
We have  $\GH(\Gamma, Y)$ and
\[
(\forall^{\st}\alpha^{1})(\forall N^{0}, M^{0}\in \Omega)(\forall s^{0})(\overline{\alpha}M=_{0}\overline{s*00\dots}M\di \Gamma(Y, s)=_{0}G(Y, s, N) ).
\]
\end{princ}
Recall the definition of $\NPC(Y)$ from Section \ref{laiba}.  
\begin{theorem}\label{copycat1339}
In $\P$, we have $ (\forall^{\st}Y^{2}\in C)(\forall \Gamma^{3})[\textup{\textsf{GHS}}_{\ns2}(Y, \Gamma)\di \ref{druk2}(Y)]$.  
%
\end{theorem}
\begin{proof}
In a nutshell, the proof of Theorem \ref{copycat} goes through with minor modifications.  In more detail, fix standard $Y^{2}\in C$ and any $\Gamma^{3}$ such that $\GHS_{\ns2}(Y, \Gamma)$.  
Applying underspill to the second conjunct of the latter, we obtain a variation of \eqref{drink}, namely the following formula:
\[
(\forall^{\st}\alpha^{1})(\exists^{\st}K^{0})(\forall N^{0}, M^{0}\geq K)(\forall s^{0})(\overline{\alpha}M=_{0}\overline{s*00\dots}M\di \Gamma(Y, s)=_{0}G(Y, s, N) ).
\]
Applying $\HAC_{\INT}$, there is $\xi^{2}$ witnessing the existential quantifier, bearing in mind Remark \ref{simply}. 
Now note that in the series of equalities involving \eqref{dung2}, all equalities resulting from \eqref{drink} only involve $\Gamma(Y, \cdot)$ and $G(Y, \cdot)$.  
Hence, the proof of Theorem~\ref{copycat} goes through in this case, but with $\xi(\alpha)$ instead of $\Xi(Y, \alpha)$.  
\end{proof}
Now define \textsf{GHS}$(Y, \Psi, \Gamma)$ as $\GHS(\Psi, \Gamma)$ without the quantifier `$(\forall Y^{2}\in C)$'. 
\begin{cor}[Term Extraction IV]\label{TvEE}
From the proof in Theorem \ref{copycat1339}, a term $t$ can be extracted such that \textsf{\textup{E-PA}}$^{\omega*}+\QFAC^{1,0}$ proves for all $ \Gamma^{3}, \Psi^{2}, Y^{2}\in C$ that
\begin{align}\label{tinkelE}
\big[\GH(\Gamma, Y)\wedge \textsf{\textup{GHS}}(Y,\Psi, \Gamma)  \big] \di \textup{\textsf{PCM}}(Y, t(Y, \Psi)). 
\end{align}
\end{cor}
\begin{proof}
Analogous to the proof of Corollary \ref{TE}.  
\end{proof}
Note that \eqref{tinkelE} expresses that if for $Y^{2}\in C$ we can approximate the Gandy-Hyland functional at $Y$ `uniformly' via $G$ and $\Psi$, then $t(Y, \Psi)$ is a modulus of pointwise continuity for $Y$.  A kind of converse was obtained in Corollary~\ref{boei}.    

%

\medskip

In light of \eqref{consversion} and\footnote{Following the proof of Corollary \ref{forlk}, Feferman's search operator may be defined in terms of a modulus-of-continuity functional, i.e.\ there is 
a version of \eqref{tinkel} with consequent $\MU(u(\Gamma, \Psi))$, where the term $u$ would however contain choice functionals from $\QFAC^{1,0}$.} \eqref{tinkel}, there are terms from G\"odel's $\textsf{T}$ expressing the (approximations of the) Gandy-Hyland functional in terms of Feferman's search operator, and vice versa. 
As it turns out, there is also a recent \emph{model-theoretic} characterisation of the Gandy-Hyland functional and arithmetical comprehension, 
namely \cite{longmann}*{Theorem~9.5.4, p.\ 460}, which expresses that: 
\begin{center}
\emph{The totality of the Gandy-Hyland functional in a \(computationally closed\) model is equivalent to that model satisfying arithmetical comprehension. \hfill \textsf{\textup{(LN)}}}
\end{center}
Consequently, it is a natural question (due to Dag Normann) whether our results are related to the aforementioned model-theoretic result.  
While our above results regarding the $\Gamma$-functional and arithmetical comprehension (as in Feferman's search operator) bear \emph{some} resemblance to \textsf{(LN)}, they are not really satisfactory.  
On the other hand, the latter deals with partial functionals, and how \emph{would} one express partiality in a system like $\P$ \emph{where all functionals are total anyway}?         
We discuss these matters in the next section where we also improve upon the previous results.  

\subsection{From Nonstandard Analysis to relative computability IV}\label{quatre}
\subsubsection{Introduction}
In this section, we study the equivalences between a version of $\NPC$, principles involving the $\Gamma$-functional, and $\paai$.  
From these equivalences, we obtain relative computability results for arithmetical comprehension, the $\Gamma$-funtional, and a modulus-of-continuity functional, some quite similar to \textsf{{(LN)}}.       
To this end, we shall work with the \emph{associates} of continuous functionals, rather than the functionals themselves.  

\medskip
 
We introduce the notion of associate in Section \ref{firlbi}, and prove an equivalence between $\NPC$ for associates and $\paai$.  From this equivalence, we obtain an effective equivalence between arithmetical comprehension as in $(\mu^{2})$, and a modulus-of-continuity functional for associates.  
In section~\ref{GHASS}, we prove an equivalence between $\paai$ and various statements regarding the Gandy-Hyland functional defined on associates.
From these nonstandard equivalences, we obtain various relative computability results regarding the Gandy-Hyland functional and arithmetical comprehension as in $(\mu^{2})$.  
As we will see, our final result is rather close in spirit to \textsf{(LN)}.  

\subsubsection{Continuity and associates}\label{firlbi}
In this section, we introduce the notion of associate and prove a first equivalence involving $\paai$ and $\NPC$ \emph{for associates}.  From this equivalence, we obtain an effective equivalence between arithmetical comprehension as in $(\mu^{2})$, and a modulus-of-continuity functional \emph{for associates}.  

\medskip

We introduce the definition of associate from \cite{kohlenbach4}*{Def.\ 4.3}; See also \cite{longmann}*{\S8.2.1}.  
\begin{definition}[Associate]\label{defke}
The function $\gamma^{1}$ is an \emph{associate} of $Y^{2}\in C$ if:
\begin{enumerate}[(i)]
\item $(\forall \beta^{1})(\exists k^{0})\gamma(\overline{\beta} k)>0$,\label{defitem2}
\item $(\forall \beta^{1}, k^{0})(\gamma(\overline{\beta} k)>0 \di Y(\beta)+1=_{0}\gamma(\overline{\beta} k))$.\label{defitem}
\end{enumerate}
We assume an associate $\gamma^{1}$ to be a \emph{neighbourhood function} (See \cite{kohlenbach4}*{\S4}), i.e.\ 
\be\label{good}
(\forall \sigma^{0}, \tau^{0}, n^{0})\big((\sigma \preceq\tau \wedge \gamma(\overline{\sigma}n)>0 ) \di \gamma(\sigma)=_{0}\gamma(\tau)\big),   
\ee
where `$\sigma\preceq \tau$' is `$|\sigma|\leq |\tau|\wedge (\forall i<|\sigma|)(\sigma(i)=\tau(i))$', i.e.\ $\sigma$ is an initial segment of $\tau$.  
\end{definition}
\noindent
We now argue why working with associates, rather than continuous functionals, is natural \emph{in our context}.    
Recall the following fragment of the axiom of choice.        
\begin{definition}[$\QFAC^{1,0}$] For internal and quantifier-free $\varphi_{0}$, we have
\be
(\forall x^{1})(\exists y^{0})\varphi_{0}(x, y)\di (\exists F^{2})(\forall x^{1})\varphi_{0}(x, F(x)).
\ee
\end{definition}
Applying $\QFAC^{1,0}$ to item \eqref{defitem2} in Definition \ref{defke}, the latter gives rise to a continuous functional $Y^{2}$ by putting $Y(\alpha):=\gamma(\overline{\alpha}F(\alpha))-1$.  Hence, associates give rise to continuous type two functionals, modulo $\QFAC^{1,0}$.  
Furthermore, the latter is a very weak principle, as established in \cite{kohlenbach2}*{\S2}.   

\medskip

Secondly, as noted above, the proof of \cite{kohlenbach4}*{Prop.\ 4.4} contains an explicit definition for obtaining an associate from a functional and its modulus of pointwise continuity.  Hence, in the presence of a modulus-of-continuity functional (as in Section \ref{deballen}) or if a modulus is assumed to be given (as in Section \ref{laiba}), working with associates rather than the continuous functionals themselves, amounts to the same.  

\medskip

Thirdly, the logical framework for \emph{Reverse Mathematics} (\cite{simpson2}) is second-order arithmetic, and one is thence forced to work with associates (called \emph{codes} by Simpson in \cite{simpson2}) to study e.g.\ continuous functionals on Baire or Cantor space, or $\R$ (See also \cite{kohlenbach4}*{Prop.\ 4.4}).  The development of Reverse Mathematics does not 
seem to be hampered by the use of associates       

\medskip

Finally, we note that associates play an important role in higher-order computability theory (See e.g.\ \cite{longmann}*{\S8.2.1}), i.e.\ they are of independent interest besides the above pragmatic motivations.  

\medskip

In light of the previous observations, it seems acceptable to work with associates directly, \emph{in the context of this paper}.  
Thus, we may introduce the following.  
\begin{nota}\label{laaaaate}\rm
We denote `$\gamma^{1}\in C$' the first item of Definition \ref{defke} plus the requirement on neighbourhood functions \eqref{good}.  
Then $(\forall^{\st}\gamma^{1}\in C)\varphi(\gamma)$ is short for 
\be\label{asin}
(\forall \gamma^{1})\big(\st_{1}(\gamma)\di\big[ \gamma^{1}\in C\di \varphi(\gamma)  \big]\big).  
\ee
Note that \emph{no} mention of $(\gamma^{1}\in C)^{\st}$ is made, and that the formula in square brackets in \eqref{asin} is internal if $\varphi(\gamma^{1})$ is.
Furthermore, we denote the `value of the associate $\gamma^{1}\in C$ at $\alpha^{1}$' by `$\gamma(\alpha)$', which is understood to be $\gamma(\overline{\alpha}N)-1$, assuming the latter is at least zero, i.e.\ for large enough $N$.  An equality `$\gamma(\alpha)=_{0}m^{0}_{0}$' is then interpreted as $(\forall n^{0})\big( \gamma(\overline{\alpha}n)>0\di \gamma(\overline{\alpha}n)=_{0}m_{0}+1)   \big)$, which is not quantifier-free.  
\end{nota}
The previous notations are in line with those used in Reverse Mathematics, as can be gleaned from \cite{simpson2}*{II.6.1}.  
With these conventions in place, we can introduce a nonstandard continuity principle on associates, as follows.  
\be\tag{$\NPC^{\mathbb{a}}$}
(\forall^{\st}\gamma^{1}\in C, \alpha^{1})(\forall \beta^{1})(\alpha\approx_{1}\beta\di \gamma(\alpha)=_{0}\gamma(\beta)),
\ee
where the final equality is not quantifier-free by Notation \ref{laaaaate}.  
\begin{theorem}\label{copycatje}
The system $\P$ proves $\NPC^{\mathbb{a}}\asa \paai$.   
\end{theorem}
\begin{proof}
For the implication $\paai\di \NPC^{\a}$, fix standard $\gamma^{1}\in C$, standard $\alpha^{1}$, and any $\beta^{1}\approx_{1}\alpha^{1}$.  Now consider $(\exists k^{0})\gamma(\overline{\alpha} k)>0$, implying $(\exists^{\st} k_{0}^{0})\gamma(\overline{\alpha} k_{0})>0$ by $\paai$.  
Since $\alpha\approx_{1}\beta$, we have $\overline{\beta} k_{0}=_{0^{*}}\overline{\alpha} k_{0}$ for the latter $k_{0}$.  
By extensionality, we have $\gamma(\overline{\alpha} k_{0})=_{0}\gamma(\overline{\beta} k_{0})>0$; This implies $\gamma(\alpha)=\gamma(\beta)$, and $\NPC^{\a}$ follows.     

\medskip

Working in $\P+\NPC^{\a}$, suppose $\paai$ is false, i.e.\ there is standard $h^{1}$ such that $(\forall^{\st}n)h(n)=0$ and $(\exists m_{0})h(m_{0})\ne 0$.  Define standard $\gamma^{1}$ as follows:
\be\label{qast}
\gamma(\sigma):=
\begin{cases}
1+ \sigma\big((\mu n\leq |\sigma|)[h(n)\ne 0]\big) & (\exists n\leq |\sigma|)[h(n)\ne 0] \\
0 & \textup{otherwise}
\end{cases}.
\ee
Clearly, $\gamma^{1}\in C$ and $\NPC^{\a}$ implies that the latter is nonstandard continuous.  However, for $\beta_{M}:=(\overline{00\dots}M)*(MM\dots)$, we note that $\beta_{0}$ is standard, and that $\beta_{0}\approx_{1} \beta_{m_{0}}$ and $\gamma(\beta_{0})=0\ne m_{0}= \gamma(\beta_{m_{0}})$ if $h(m_{0})\ne 0$.  This contradiction yields $\NPC^{\a}\di \paai$, and we are done.  
\end{proof}
%
%
Recall $\MU(\mu)$ introduced in Section \ref{seqtor} and let $\MPC^{\a}(\Psi^{2})$ be 
\[
(\forall \gamma^{1} \in C)(\forall f^{1}, g^{1}\leq_{1}1)\big[\overline{g}\Psi(\gamma, f)=_{0} \overline{f}\Psi(\gamma, f)\di \gamma(f)=_{0}\gamma(g)\big].  
\]
The formula in square brackets in $\MPC^{\a}(\Psi)$ is not quantifier-free due to `$\gamma(f)=_{0}\gamma(g)$'.  We have the following relative computability result.  
\begin{cor}[Term Extraction V]\label{thelastsheeversawofhim}
From the proof of $ \NPC^{\a}\asa\paai$ in $\P$, terms $s,t$ can be extracted such that $\textsf{\textup{E-PA}}^{\omega*}+\QFAC^{1,0}$ proves
\be\label{frood1440}
(\forall \mu^{2})\big[\textsf{\MU}(\mu)\di \MPC^{\a}(s(\mu)) \big] \wedge (\forall \Psi^{2})\big[ \MPC^{\a}(\Psi)\di  \MU(u(\Psi))  \big].
\ee
\end{cor}
\begin{proof}
To obtain a normal form for $\NPC^{\a}$, proceed in the same way as for $\NUC$ and \eqref{kunt3}.  In particular the normal form of $\NPC^{\a}$ is
\be\label{NFA}
(\forall^{\st}\gamma^{1}\in C, f^{1})(\exists^{\st}N^{0})(\forall g^{1})\big[\overline{f}N=_{0}\overline{g}N\di \gamma(f)=_{0}\gamma(g)\big],
\ee
where the formula in square brackets is \emph{internal} and \emph{not quantifier-free} by Notation~\ref{laaaaate}.  
The normal form for $\paai$ is obvious, namely:
\be\label{nfpaai}
(\forall^{\st}f^{1})(\exists^{\st}n^{0})\big[(\exists m^{0})f(m)\ne0 \di (\exists i^{0}\leq_{0} n)f(i)\ne0  \big].
\ee
Now apply the template in Remark \ref{firliborn} to $\NPC^{\a}\asa \paai$.  
\end{proof}
Finally, we discuss the conceptual meaning of $\NPC^{\a}$.  By Notation \ref{laaaaate}, a \emph{type one} associate $\gamma^{1}\in C$ is meant to `simulate' or `represent' a continuous \emph{type two} functional.  By Definition \ref{debs}, \emph{standard} functionals have standard output for standard input in $\P$;  Thus, a natural question is whether a \emph{standard} associate $\gamma^{1}\in C$ has standard output $\gamma(\alpha)$ for standard input $\alpha^{1}$, i.e.\ whether a \emph{standard} associate also simulates a \emph{standard} type two functional.  By the proof of Theorem \ref{copycatje}, one requires $\paai$ or $\NPC^{\a}$ to guarantee the `expected' behaviour that standard associates have standard output for standard input.  Indeed, $\gamma^{1}\in C$ as in \eqref{qast} yields \emph{nonstandard} output for certain \emph{standard} inputs, assuming $\neg\paai$.  

\subsubsection{The Gandy-Hyland functional and associates}\label{GHASS}
In this section, we shall study the connection between arithmetical comprehension and the Gandy-Hyland functional \emph{defined on associates}.  
Now, in the previous section, we observed that \emph{type one associates} may be viewed as \emph{type two functionals};  In particular, for $\gamma^{1}\in C$ and $\alpha^{1}$, it makes sense to apply the former to the latter as in `$\gamma(\alpha)$', despite the type mismatch.  
Similarly, we now define how one applies \emph{type three} functionals (like the Gandy-Hyland functional) to type one associates, again despite the apparent type mismatch.  
\begin{nota}\label{riffer}\rm
In the presence of $\QFAC^{1,0}$, $\gamma(\alpha)$ equals $\gamma(\overline{\alpha}F(\alpha))-1$ for $\gamma^{1}\in C$, where $F$ originates from the former choice axiom applied to $(\forall \alpha^{1})(\exists N^{0})\gamma(\overline{\alpha}N)>0$.  In this way, we define for $\Lambda^{3}$ and $\gamma^{1}\in C$, the application of the former functional to the latter sequence as 
$\Lambda(\gamma):= \Lambda\big((\lambda \alpha^{1})[\gamma(\overline{\alpha}F(\alpha))-1] \big)$, where $F$ is the aforementioned choice functional from $\QFAC^{1,0}$.  Similarly, for a formula $A(Y^{2})$, we shall use the formula $(\forall \gamma^{1}\in C)A(\gamma)$ as shorthand for the following formula:
\be\label{kkkunt}
(\forall \gamma^{1}) (\forall F^{2})\Big[ (\forall \alpha^{1})\gamma(\overline{\alpha}F(\alpha))>0\di   A\big((\lambda \alpha^{1})[\gamma(\overline{\alpha}F(\alpha))-1]\big)\Big], 
\ee
where no type mismatch occurs.  
Similar to the convention involving \eqref{asin}, the formula `$(\forall^{\st} \gamma^{1}\in C)A(\gamma)$' is the formula \eqref{kkkunt} 
with `$(\forall^{\st}\gamma^{1})$' instead of `$(\forall \gamma^{1})$'.
\end{nota}
As will become clear, mathematics practice does not change much when working with associates;  This has been previously observed in the development of Reverse Mathematics (See e.g.\ \cite{simpson2}*{I.4, p.\ 15}).  

\medskip

First of all, we study the following principle regarding the Gandy-Hyland functional and associates:  
\begin{align}\tag{$\GH_{\ns}^{\mathbb{a}}$}
(\exists^{\st}\Gamma^{3})\big[(\forall^{\st}\gamma^{1}\in C)(\forall s^{0})&\GH(\Gamma, \gamma, s)\\
&\wedge (\forall^{\st}\gamma^{1}\in C,s^{0})(\forall N\in \Omega)( \hat{G}(\gamma,s,N)=\Gamma(\gamma,s)) \big],\notag
\end{align}
where $\hat{G}$ is $G$ with $Y(s*11\dots)$ rather than $Y(s*00\dots)$ in the first case of \eqref{small}.    

\medskip

In light of Theorem \ref{drag} and Theorem \ref{copycatje}, it is easy to obtain a proof of $\paai \di \GH_{\ns}^{\a}$ in $\P+\STP$.
The more interesting reversal is now as follows.
\begin{theorem}\label{mukh}
The system $\P$ proves $\GH_{\ns}^{\mathbb{a}}\di \paai$.
\end{theorem}
\begin{proof}
Working in $\P+\GH_{\ns}^{\a}$, we show that every standard $\gamma^{1}\in C$ is nonstandard continuous, 
implying $\NPC^{\a}$ and thus $ \paai$ by Theorem \ref{copycatje}.  
If we have $(\forall^{\st}\alpha^{1})(\exists^{\st}N)\gamma(\overline{\alpha}N)>0$ for $\gamma^{1}\in C$, then the latter is nonstandard continuous.  
We now derive a contradiction from $(\exists^{\st}\alpha_{0}^{1})(\forall^{\st}N)\gamma_{0}(\overline{\alpha}N)=0$ for some standard $\gamma_{0}^{1}\in C$.  
Fix such $\alpha_{0},\gamma_{0}$ and define the standard function $\gamma^{1}\in C$ as follows:
\be\label{qast2}
\gamma(\sigma):=
\begin{cases}
1+ \sigma((\mu n\leq |\sigma|)\gamma_{0}(\overline{\alpha_{0}}n)\ne 0) & (\exists n\leq |\sigma|)\gamma_{0}(\overline{\alpha_{0}})\ne 0 \\
0 & \textup{otherwise}
\end{cases}.
\ee   
Note that $\gamma(\alpha)=\alpha(m_{0})$ where $m_{0}$ is the least $m$ such that $\gamma_{0}(\overline{\alpha}m)\ne 0$.
Bearing in mind Notation \ref{riffer}, we compute $\Gamma(\gamma, \langle \rangle)$ and $\hat{G}(\gamma, \langle \rangle, m_{0})$ and observe that they are different.  This contradiction yields $\NPC^{\a}$, and $\paai$ by Theorem \ref{copycatje}.     
To prove that $\Gamma(\gamma, \langle \rangle)\ne\hat{G}(\gamma, \langle \rangle, m_{0})$, define $\beta_{0}^{1}$ as $\big(0*(\lambda n)\Gamma(\gamma, n+1)\big)$;  By the definition of $\gamma$:
\[
\Gamma(\gamma, \langle\rangle)=\gamma(0*(\lambda n)\Gamma(\gamma, n+1))=\gamma(\beta_{0})=  \beta_{0}(m_{0})=\Gamma(\gamma, m_{0}). 
\]
Similarly, define $\beta_{1}^{1}$ as $\big(m_{0}*0*(\lambda n)\Gamma(\gamma, m_{0}*(n+1))\big)$ and note that:
\begin{align*}
\Gamma(\gamma, m_{0})=\gamma(m_{0}*0*(\lambda n)\Gamma(\gamma, m_{0}*(n+1)))=\gamma(\beta_{1})=\beta_{1}(m_{0})=\Gamma(\gamma, m_{0}*(m_{0}-1)).
\end{align*}
Hence, $\Gamma(\gamma, \langle\rangle)=\Gamma(\gamma, m_{0}*(m_{0}-1))$, and in the same way one obtains that
\[
\Gamma(\gamma, \langle\rangle)=\Gamma(\gamma, m_{0}*(m_{0}-1)*\dots* 1)=\gamma\big(m_{0}*(m_{0}-1)*\dots* 1*0* (\lambda n)(\dots)\big)=0.
\]
Secondly, in exactly the same way, we have for $\sigma_{0}^{0}:=m_{0}*(m_{0}-1)*\dots* 1$ that
\begin{align*}
\hat{G}(\gamma, \langle\rangle, m_{0})=\hat{G}(\gamma, \langle m_{0}\rangle, m_{0})
&=\hat{G}(\gamma, \langle m_{0},m_{0}-1\rangle, m_{0})\\
&=\hat{G}(\gamma, \langle m_{0},m_{0}-1, m_{0}-2\rangle, m_{0})\\
&=\dots\\
&=\hat{G}(\gamma, \langle m_{0},m_{0}-1, m_{0}-2,\dots,1\rangle, m_{0})\\
&=\hat{G}(\gamma, \sigma_{0}, m_{0})\\
&=\gamma\big(\sigma_{0}*11\dots\big)=(\sigma_{0}*11\dots)(m_{0})=1,
\end{align*}
where $m_{0}$ is again the least $m$ such that $h(m)\ne 0$. 
\end{proof}
The previous nonstandard theorem gives rise to the following term extraction corollary, for which we need:  
\begin{align}\tag{$\GH^{\a}(\Gamma)$}
&(\forall \gamma^{1}\in C, s^{0})\big[\Gamma(\gamma,s^{0})= \gamma\big(s*0* (\lambda n^{0})\Gamma(\gamma, s*(n+1))\big)\big],\\
\tag{$\GHS^{\a}(\Psi, \Gamma)$}
&(\forall \gamma^{1} \in C, s)(\forall N\geq \Psi(\gamma, s))\big(\hat{G}(\gamma, s, N)=_{0}\Gamma(\gamma, s) \big).
\end{align}
\vspace{-0.75cm}
\begin{cor}[Term extraction VI]\label{figger}
From the proof of $ \GH_{\ns}^{\a}\di\paai$ in $\P$, a term $t$ can be extracted such that $\textsf{\textup{E-PA}}^{\omega*}+\QFAC^{1,0}$ proves 
\be\label{frood1447}
(\forall \Psi^{2}, \Gamma^{3})\big[ [\GH^{\a}(\Gamma)\wedge \GHS^{\a}(\Psi, \Gamma)]\di  \MU(t(\Psi, \Gamma))  \big].
\ee
\end{cor}
\begin{proof}
Analgous to the proof of Corollary \ref{thelastsheeversawofhim};  We provide a more detailed sketch to show that Notation \ref{riffer} does not interfere with term extraction as in Remark~\ref{firliborn}.  To this end, note that the second conjunct of $\GH_{\ns}^{\a}$ is 
\[
(\forall^{\st}\gamma^{1})(\forall F^{2})\Big[(\forall \alpha^{1})(\gamma(\overline{\alpha}F(\alpha))>0) \di   (\forall^{\st} s^{0})(\forall N\in \Omega)( \hat{G}(\gamma_{F},s,N)=\Gamma(\gamma_{F},s))\Big],
\]
where $\gamma_{F}^{2}$ is $(\lambda \alpha^{1})[\gamma(\overline{\alpha}F(\alpha))-1]$.  Clearly, the quantifiers `$(\forall^{\st} s^{0})(\forall N\in \Omega)$' can be pushed outside to obtain a formula of the form $(\forall^{\st}\gamma^{1}, s^{0})(\forall N\in \Omega)A(\gamma, s, N, \Gamma)$, where $A$ is internal.  Applying underspill, we obtain 
\be\label{korfen}
(\forall^{\st}\gamma^{1}, s^{0})(\exists^{\st}n^{0})(\forall N\geq n)A(\gamma, s, N, \Gamma),
\ee
which is the normal form of the second conjunct of $\GH_{\ns}^{\a}$.  The normal form of the first conjunct of the latter is now obtained similarly.  
The normal form of $\paai$ is \eqref{nfpaai}, and applying the template in Remark~\ref{firliborn} yields \eqref{frood1447}.  
In particular, in the course of applying step \eqref{ted} of the template, \eqref{korfen} is transformed into
\[
(\exists^{\st}\Psi^{2})(\forall \gamma^{1}, s^{0})(\forall N\geq \Psi(\gamma, s))A(\gamma, s, N, \Gamma),
\]
and writing out $A$ in full again, we obtain:
\[
(\exists^{\st}\Psi^{2})(\forall \gamma^{1}, s^{0})(\forall N\geq \Psi(\gamma, s))(\forall F^{2})\big[(\forall \alpha^{1})(\gamma(\overline{\alpha}F(\alpha))>0) \di   \hat{G}(\gamma_{F},s,N)=\Gamma(\gamma_{F},s)\big].
\]
Rearranging the universal quantifiers, the previous formula is $(\exists^{\st}\Psi^{2})\GHS^{\a}(\Psi, \Gamma)$ by Notation \ref{riffer}, which is exactly as required for obtaining \eqref{frood1447}. 
\end{proof}
By studying the proof of Theorem~\ref{mukh} in more detail, one observes that `$\Gamma$ is standard' in $\GH_{\ns}^{\a}$ is superfluous.  
Repeating the proof of Corollary~\ref{figger} with this modification, one obtains \eqref{frood1447} where the term only depends on $\Psi$.  

\medskip

Next, we study a variation of $\GH_{\ns}^{\a}$ based on \cite{samzoo}.  In the latter, a number of effective equivalences between arithmetical comprehension and uniform theorems from the \emph{Reverse Mathematics zoo} (\cite{damirzoo}) are extracted from \emph{nonstandard} equivalences involving $\paai$ and \emph{standard extensionality}.  
Thus, we are led to the following:
\begin{align}\tag{$\GH_{\ns2}^{\mathbb{a}}$}
(\exists^{\st}\Gamma^{3})\big[(\forall^{\st}\gamma^{1}\in C)(\forall s^{0})\GH(\Gamma, \gamma, s)\wedge (\forall^{\st}\gamma^{1},\eps^{1} \in C,s^{0})(\gamma\approx_{1}\eps\di \Gamma(\gamma, s)=_{0}\Gamma(\eps, s) \big],\notag
\end{align}
which expresses that the Gandy-Hyland functional is \emph{standard extensional} similar to $\eqref{EXT}^{\st}$ defined in Remark \ref{equ}.  Note that $\paai$ immediately yields standard extensionality $\eqref{EXT}^{\st}$ from `usual' extensionality $\eqref{EXT}$ for standard functionals of type two.  The reverse implication is again more interesting.    
\begin{theorem}\label{mukh2}
The system $\P$ proves $\GH_{\ns2}^{\mathbb{a}}\di \paai$.
\end{theorem}
\begin{proof}
As in the proof of Theorem \ref{mukh}, suppose $(\exists^{\st}\alpha_{0}^{1})(\forall^{\st}N)\gamma_{0}(\overline{\alpha}N)=0$ for some fixed standard $\gamma_{0}^{1}\in C$.
Let $\gamma^{1}\in C$ be as in \eqref{qast2} and recall that $\Gamma(\gamma, \langle\rangle)=0$ by the proof of Theorem~\ref{mukh}.  
Now define standard $\eps^{1}\in C$ as:  
\be\label{qast333}
\eps(\sigma):=
\begin{cases}
2+ \sigma((\mu n\leq |\sigma|)\gamma_{0}(\overline{\alpha_{0}}n)\ne 0) & (\exists n\leq |\sigma|)\gamma_{0}(\overline{\alpha_{0}})\ne 0 \\
0 & \textup{otherwise}
\end{cases}.
\ee   
Note that $\eps(\alpha)=1+\alpha(m_{0})$ where $m_{0}$ is the least $m$ such that $\gamma_{0}(\overline{\alpha}m)\ne 0$.
Clearly, we have that $\gamma\approx_{1}(00\dots)\approx_{1} \eps$, while at the same time we can compute: 
\be\label{krazy}
\Gamma(\eps, \langle\rangle)=\eps(0*(\lambda n)\Gamma(\eps, n+1))=1+\Gamma(\eps, m_{0})\ne0=\Gamma(\gamma,\langle\rangle),  
\ee
where $m_{0}$ is again the least $m$ such that $\gamma_{0}(\overline{\alpha}m)\ne 0$.  
This contradiction yields $\NPC^{\a}$ and hence $\paai$ by Theorem \ref{copycatje}.  
\end{proof}
Recall the notion of `extensionality functional' from Section \ref{tempsde} and define:  
\be\tag{$\EXT^{\a}(\Xi^{2}, \Lambda^{3})$}
(\forall \gamma^{1}\in C, \eps^{1}\in C)\big(\overline{\gamma}\Xi(\gamma,\eps)=_{0}\overline{\eps}\Xi(\gamma,\eps)\di \Lambda(\gamma)=_{0}\Lambda(\eps)\big).
\ee
\begin{cor}[Term extraction VII]\label{figger2}
From the proof of $ \GH_{\ns2}^{\a}\di\paai$ in $\P$, a term $t$ can be extracted such that $\textsf{\textup{E-PA}}^{\omega*}+\QFAC^{1,0}$ proves  that
\be\label{frood1448}
(\forall \Xi^{2}, \Gamma^{3})\big[ [\GH^{\a}(\Gamma)\wedge (\forall s^{0})\EXT^{\a}(\Xi(\cdot, s), \Gamma(\cdot, s))]\di  \MU(t(\Xi, \Gamma))  \big].
\ee
\end{cor}
\begin{proof}
Analogous to the proof of Corollary \ref{figger}.  We again show that Notation \ref{riffer} does not cause problems for term extraction as in Remark \ref{firliborn}.    
First of all, the second conjunct of $\GH^{\a}_{\ns}2$ is, by Notation \ref{laaaaate} and Remark \ref{equ}:
\begin{align}
 (\forall^{\st}\gamma^{1},\eps^{1})(\forall F^{2}, H^{2})\Big[(\forall& \alpha^{1})(\gamma(\overline{\alpha}F(\alpha))>0)\wedge (\forall \alpha^{1})(\eps(\overline{\alpha}H(\alpha))>0)\label{omh} \\
 & \di (\forall^{\st}s^{0})(\exists^{\st}N)(\overline{\gamma}N=_{0}\overline{\eps}N\di \Gamma(\gamma_{F}, s)=_{0}\Gamma(\eps_{F}, s) \Big],\notag
\end{align}
where $\gamma_{F}^{2}$ is $(\lambda \alpha^{1})[\gamma(\overline{\alpha}F(\alpha))-1]$, and similar for $\eps_{F}^{2}$.
Now, \eqref{omh} can be brought into the form $ (\forall^{\st}\gamma^{1},\eps^{1}, s^{0})(\forall F^{2}, H^{2})(\exists^{\st}N)A(\gamma, \eps, F, H, N,s)$, where $A$ is internal.  Applying idealisation \textsf{I} as in Remark \ref{simply} yields: 
\be\label{finito}
 (\forall^{\st}\gamma^{1},\eps^{1}, s^{0})(\exists^{\st}N)(\forall F^{2}, H^{2})A(\gamma, \eps, F, H, N,s).  
\ee
After applying step \eqref{ted} from the template in Remark \ref{firliborn}, \eqref{finito} becomes
\[
(\exists^{\st}\Phi^{2})(\forall s^{0})(\forall \gamma^{1},\eps^{1})(\forall F^{2}, H^{2})A(\gamma, \eps, F, H, \Phi(s, \gamma, \eps),s),   
\]
which can be brought into $(\exists^{\st}\Xi^{2})(\forall s^{0})\EXT^{\a}(\Xi(\cdot,s ), \Gamma(\cdot, s))$, by Notation \ref{riffer}. 
 A normal form for the first conjunct of $\GH_{\ns2}^{\a}$ is now straightforward, while a normal form for $\paai$ is given by \eqref{nfpaai}.  The template from Remark~\ref{firliborn} is now easily seen to yield the relative computability result \eqref{frood1448}.  
\end{proof}
Following the proof of Theorem~\ref{mukh2} in detail, it becomes clear that the condition `$\Gamma$ is standard' is superfluous, implying that the term in \eqref{frood1448} only depends on $\Xi$.  

\medskip

The previous results are not satisfactory since extensionality for associates as in $\EXT^{\a}$ does not `fully' treat $\gamma^{1}, \eps^{1}\in C$ as type two functionals.  However, the proof of Theorem \ref{mukh2} does provide us with an interesting way forward;  In particular, it is easy to compute that $\Gamma(\eps, \langle\rangle)$ from \eqref{krazy} is nonstandard.  However, this means that $\Gamma$ is \emph{nonstandard for standard inputs}, while standard functionals (should) have standard output for standard input.      
Thus, we are led to the final variation of $\GH_{\ns}^{\a}$:
\begin{align}\tag{$\GH_{\ns3}^{\mathbb{a}}$}
(\exists \Gamma^{3})&\big[(\forall^{\st}\gamma^{1}\in C)(\forall s^{0})\GH(\Gamma, \gamma, s)\wedge (\forall^{\st} \eps^{1}\in C)(\forall^{\st}t^{0})(\st_{0}(\Gamma(\eps,t)))\big].
\end{align}
which merely expresses that the Gandy-Hyland functional exists and is standard for standard input.  
As expected, $\paai$ implies $\GH_{\ns3}^{\a}$ but
the reverse implication is again more interesting.    
\begin{theorem}\label{mukh333}
The system $\P$ proves $\GH_{\ns3}^{\mathbb{a}}\di \paai$.
\end{theorem}
\begin{proof}  Assume $\GH_{\ns3}^{\a}$;  
As in the proof of Theorem \ref{mukh2}, suppose $(\exists^{\st}\alpha_{0}^{1})(\forall^{\st}N)\gamma_{0}(\overline{\alpha}N)=0$ for some fixed standard $\gamma_{0}^{1}\in C$, and define standard 
$\eps^{1}\in C$ as in \eqref{qast333}.  
Again note that $\eps(\alpha)=1+\alpha(m_{0})$ where $m_{0}$ is the least $m$ such that $\gamma_{0}(\overline{\alpha}m)\ne 0$.
Now compute $\Gamma(\eps, \langle\rangle)$ as follows: $\Gamma(\eps, \langle\rangle)=\eps(0*(\lambda n)\Gamma(\eps, n+1))=1+\Gamma(\eps, m_{0})$ and 
\begin{align*}
\Gamma(\eps, m_{0})=\eps(m_{0}*0*(\lambda n)\Gamma(\eps, m_{0}*(n+1)))=1+\Gamma(\gamma, m_{0}*(m_{0}-1))
\end{align*}
Similarly, we have $\Gamma(\eps, m_{0}*(m_{0}-1))=1+\Gamma(\eps, m_{0}*(m_{0}-1)*(m_{0}-2))$, and hence $\Gamma(\eps, \langle\rangle)=2+\Gamma(\eps, m_{0}*(m_{0}-1)*(m_{0}-2))$.   Ultimately, we obtain 
\be\label{seealso}
\Gamma(\eps, \langle\rangle)=\dots=(m_{0}-1)+\Gamma(\eps, m_{0}*(m_{0}-1)*\dots* 1), 
\ee
by applying the same procedure $m_{0}-1$ times.  However, $\Gamma(\eps, \langle\rangle)$ is thus nonstandard, and this contradiction yields $\NPC^{\a}$, and $\paai$ follows by Theorem \ref{copycatje}.  
\end{proof}
\begin{cor}[Term extraction VIII]\label{figger333}
From the proof of $ \GH_{\ns3}^{\a}\di\paai$ in $\P$, a term $t$ can be extracted such that $\textsf{\textup{E-PA}}^{\omega*}+\QFAC^{1,0}$ proves  that
\be\label{frood14499}
(\forall  \Gamma^{3}, \Xi^{2})\Big[\big[ \GH^{\a}(\Gamma) \wedge (\forall \gamma^{1}\in C, s^{0})\big(\Xi(\gamma, s)=_{0}\Gamma(\gamma, s)\big)\big]\di  \MU(t(\Xi))  \big].
\ee
\end{cor}
\begin{proof}
Analogous to Corollary \ref{figger} and \ref{figger2}.  
We show that Notation~\ref{riffer} does not cause problems for term extraction as in Remark \ref{firliborn}.    
First of all, the second conjunct of $\GH_{\ns3}^{\a}$ is:
\[
(\forall^{\st}\eps^{1})(\forall F^{2})\Big[(\forall \alpha^{1})(\eps(\overline{\alpha}F(\alpha))>0) \di   (\forall^{\st} t^{0})(\exists^{\st}n^{0})( n=\Gamma(\eps_{F},t))\Big],
\]
where $\eps_{F}^{2}$ is $(\lambda \alpha^{1})[\eps(\overline{\alpha}F(\alpha))-1]$.  Push all standard quantifiers outside:
\[
(\forall^{\st}\eps^{1}, t^{0})(\forall F^{2})(\exists^{\st}n^{0})\Big[(\forall \alpha^{1})(\eps(\overline{\alpha}F(\alpha))>0) \di  ( n=\Gamma(\eps_{F},t))\Big],
\]
and apply idealisation \textsf{I} to obtain: 
\[
(\forall^{\st}\eps^{1}, t^{0})(\exists^{\st}m)(\forall F^{2})(\exists n^{0}\leq m)\Big[(\forall \alpha^{1})(\eps(\overline{\alpha}F(\alpha))>0) \di  ( n=\Gamma(\eps_{F},t))\Big],
\]
which is a normal form, which we abbreviate $(\forall^{\st}\eps^{1}, t^{0})(\exists^{\st}m)A(\eps, t, m, \Gamma)$, where $A$ is internal.  
A normal form for $\paai$ is \eqref{nfpaai}, which we abbreviate by $(\forall^{\st}f^{1})(\exists^{\st}n^{0})B(f, n)$, where $B$ is internal.  
Hence, $ \GH_{\ns3}^{\a}\di\paai$ implies:
\[
\Big[(\exists \Gamma^{3})\big[ \GH^{\a}(\Gamma) \wedge (\forall^{\st}\eps^{1}, t^{0})(\exists^{\st}m)A(\eps, t, m, \Gamma) \big]\Big]\di (\forall^{\st}f^{1})(\exists^{\st}n^{0})B(f, n),
\]
by strengthening the antecedent (by dropping `st' in the first conjunct of $\GH_{\ns3}^{\a}$).  
Now introduce a standard functional $\Xi^{2}$ as follows: 
\[
(\forall^{\st}\Xi^{2})(\forall \Gamma^{3})\Big[\big[ \GH^{\a}(\Gamma) \wedge (\forall^{\st}\eps^{1}, t^{0})A(\eps, t, \Xi(\eps, t), \Gamma) \big]\di (\forall^{\st}f^{1})(\exists^{\st}n^{0})B(f, n)\Big],
\]
and drop the remaining `st' \emph{in the antecedent} to yield:
\[
(\forall^{\st}\Xi^{2})(\forall \Gamma^{3})\Big[\big[ \GH^{\a}(\Gamma) \wedge (\forall \eps^{1}, t^{0})A(\eps, t, \Xi(\eps, t), \Gamma) \big]\di (\forall^{\st}f^{1})(\exists^{\st}n^{0})B(f, n)\Big].
\]
Push outside the standard quantifiers (as far as possible) to obtain
\[
(\forall^{\st}\Xi^{2},  f^{1})(\forall \Gamma^{3})(\exists^{\st}n^{0})\Big[\big[ \GH^{\a}(\Gamma) \wedge (\forall \eps^{1}, t^{0})A(\eps, t, \Xi(\eps, t), \Gamma) \big]\di B(f, n)\Big],
\]
to which we apply idealisation \textsf{I} (as in Remark \ref{simply}) to obtain
\[
(\forall^{\st}\Xi^{2}, f^{1})(\exists^{\st}n^{0})(\forall \Gamma^{3})\Big[\big[ \GH^{\a}(\Gamma) \wedge (\forall \eps^{2}, t^{0})A(\eps, t, \Xi(\eps, t), \Gamma) \big]\di B(f, n)\Big].  
\]
Applying Corollary \ref{consresultcor} now yields a term $u$ such that $\textsf{E-PA}^{\omega*}+\QFAC^{1,0}$ proves 
\[
(\forall \Xi^{2}, f^{1})(\exists n^{0}\in u(\Xi, f))(\forall \Gamma^{3})\Big[\big[ \GH^{\a}(\Gamma) \wedge (\forall \eps^{1}, t^{0})A(\eps, t, \Xi(\eps, t), \Gamma) \big]\di B(f, n)\Big].  
\]
Now define $t(\Xi, f):= \max_{i<|u(\Xi, f)|}u(\Xi, f)(i)$ and note that we have 
\[
(\forall \Xi^{2}, f^{1}, \Gamma^{3})\Big[\big[ \GH^{\a}(\Gamma) \wedge (\forall \eps^{1}, t^{0})A(\eps, t, \Xi(\eps, t), \Gamma) \big]\di B(f, t(\Xi, f))\Big],   
\]
due to the monotone behaviour of $B(f, \cdot)$.  Furthermore, `$(\forall f^{1})$' can be pushed inside to obtain that
\be\label{korkkk}
(\forall \Xi^{2}, \Gamma^{3})\Big[\big[ \GH^{\a}(\Gamma) \wedge (\forall \eps^{1}, t^{0})A(\eps, t, \Xi(\eps, t), \Gamma) \big]\di \MU(t(\Xi))\Big].
\ee
Finally, we note that $(\forall \eps^{1}, t^{0})A(\eps, t, \Xi(\eps, t)$ is implied by $(\forall \eps^{1}\in C, t^{0})(\Xi(\eps, t)=\Gamma(\eps, t))$ by Notation \ref{riffer}.  
Thus, \eqref{korkkk} implies \eqref{frood14499}, and we are done.  
\end{proof}
In conclusion, \eqref{frood14499} expresses that a term from G\"odel's $\textsf{T}$ yields arithmetical comprehension as in $(\mu^{2})$ from any functional $\Xi^{2}$ which computes the values of the Gandy-Hyland functional $\Gamma$ defined on associates.  Thus, it can be said that \eqref{frood14499} is the syntactic version of (the forward implication of) the theorem \textsf{(LN)} as in \cite{longmann}*{Theorem 9.5.4, p.\ 460}.  Of course, the latter theorem is formulated with partial functionals, while all functionals in $\P$ are total.  We show in the next section that $\P$ can `simulate' partiality \emph{relative to the standard world};  We also argue that this `standard partiality' explains the results in this section.

\section{Concluding remarks}
In this paper, we have shown that certain theorems from Nonstandard Analysis give rise to (effective) relative computability results.  
This resonates nicely with the longstanding (but speculative) claim that Nonstandard Analysis is somehow `constructive' or `effective', captured well by the quote:  
\begin{quote}
It has often been held that nonstandard analysis is highly non-constructive, thus somewhat suspect, depending as it does upon the ultrapower construction to produce a model [\dots] On the other hand, nonstandard \emph{praxis} is remarkably constructive; having the extended number set we can proceed with explicit calculations. (Emphasis original: \cite{NORSNSA}*{p.\ 31})
\end{quote}
Similar observations regarding the `constructive or effective content of Nonstandard Analysis' are made in numerous places; 
An incomplete list may be found in \cite{sambon}*{\S1}.  The results in this paper can be said to make the aforementioned speculative claim 
regarding the effective content of Nonstandard Analysis \emph{more concrete}.     

\medskip

By contrast, the following final remark is somewhat vague and speculative, but partially explains the connection between the totality of the Gandy-Hyland functional mentioned in \cite{longmann}*{Theorem~9.5.4, p.\ 460} and Corollary \ref{figger333}.
\begin{rem}[Partiality in $\P$]\rm
The class of \emph{partial computable functions} is a central object of study in computability theory (\cite{zweer}*{I.2.2}).  
As discussed in the latter, there are good reasons to study partial functions. We now discuss how $\P$ can accommodate partial functionals, 
\emph{despite all functionals being total}.  Intuitively speaking, we show that a total computable function \emph{with standard index} can output nonstandard numbers for standard input (after running for nonstandard many steps).  Such a total function may rightly be called `not total from the point of view of the standard world' in view of the basic axioms of $\P$.  More formally:     

\medskip

First of all, consider the well-known predicate `$\varphi^{A}_{e, s}(n)=m$' which intuitively states that: `the $e$-th Turing machine with oracle set $A$ and input $n$ halts after $s$ steps with output $m$' (\cite{zweer}*{Def.\ 3.8}).  Now let $e_{0}, x_{0}$ be \emph{standard} numbers and $A$ a \emph{standard} set such that $(\exists s^{0}, m^{0})[\varphi^{A}_{e_{0},s}(x_{0})=m]$, i.e.\ we say that `$\varphi^{A}_{e_{0}}(x_{0})$' is defined in the usual computability-theoretic terminology.  

\medskip

Secondly, the basic axioms of $\P$ in Definition \ref{debs} guarantee that every standard functional evaluated at a standard input returns a standard output.
By contrast, \emph{without the presence of} $\paai$, $\varphi^{A}_{e_{0}}(x_{0})$ as defined above\footnote{Assume $\neg\paai$ and let $h^{1}$ be as in the proof of Theorem \ref{copycatje};  Define $e_{0}$ as the (standard) code of the program which tests if the input $x_{0}$ satisfies $x_{0}\in A:=\{n : h(n)\ne 0\}$ and outputs $x_{0}$ if so, and repeats the previous step for $x_{0}+1$ otherwise.  Then $\varphi^{A}_{e_{0}, m_{0}}(x_{0})=m_{0}$, if $m_{0}$ is the least number such that $h(m_{0})\ne 0$, while the inputs $A, e_{0}, x_{0}$ are standard.\label{hagggot}} \emph{may well be nonstandard}.  In other words, while $\varphi^{A}_{e_{0}}(x_{0})$ is defined and all inputs are standard, the $e_{0}$-th Turing machine may well take a nonstandard number $s$ of steps to halt, with a nonstandard output $m$, as discussed in Footnote \ref{hagggot}.  

\medskip

Thirdly, in light of the previous, we are led to the following definition:  For standard $e_{0}, x_{0}, A$, 
we say that `$\varphi^{A}_{e_{0}}(x_{0})$ is \emph{standard-defined}' if $(\exists^{\st} s^{0}, m^{0})[\varphi^{A}_{e_{0},s}(x_{0})=m]$, and `standard-undefined' otherwise.  Similarly, for standard $e_{0}, A$, we say that `$\varphi_{e_{0}}^{A}$ is standard-total' if we have $(\forall^{\st}x_{0})(\exists^{\st} s^{0}, m^{0})[\varphi^{A}_{e_{0},s}(x_{0})=m]$, and `standard-partial' otherwise.  Hence, define $\psi_{e}^{A}$ as follows for fixed $M\in \Omega$:
\[
\psi_{e}^{A}(x):=
\begin{cases}
 \varphi_{e}^{A}(x) &(\exists s^{0}, m^{0}\leq M)[\varphi_{e,s}^{A}(x)=m]  \\  
M+1  & \textup{otherwise}
\end{cases}
\]
By definition, $\psi_{e}^{A}$ is total \emph{but not standard-total} in the presence of $\neg\paai$ by 
Footnote~\ref{hagggot}.  
Hence, we can in fact \emph{simulate} the concept of partiality inside $\P$ by exploiting the dichotomy between `standard' and `nonstandard'.   
Similar definitions are possible for higher-type functionals.  
\end{rem}
Finally, we arrive at the motivation for the definitions in the previous remark:  Consider the standard associates $\gamma^{1}\in C$ and $\eps^{1}\in C$ as in \eqref{qast} and \eqref{qast333};  To compute $\gamma(\alpha)$ at standard $\alpha^{1}$, one simply evaluates $\gamma(\overline{\alpha}0)$, $\gamma(\overline{\alpha}1)$, et cetera, until $N^{0}$ is found such that $\gamma(\overline{\alpha}N)>0$, and the same for $\eps^{1}\in C$.  This computation always terminates by the definition of $\eps^{1}\in C$ and $ \gamma^{1}\in C$.  However, in the presence of $\neg\paai$, this computation only terminates after a \emph{nonstandard} number of steps, i.e.\ $\gamma^{1}\in C$ and $\eps^{1}\in C$ are `standard-partial' in the above sense.  However, $\NPC^{\a}$ implies that every standard $\gamma^{1}\in C$ is standard-total (as it is nonstandard continuous), and therefore $\paai$ follows from $\NPC^{\a}$;  In fact, we have an equivalence by Theorem \ref{copycatje}.     
In short, $\NPC^{\a}$ guarantees that every standard associate is standard-total, which apparently requires $\paai$, and the latter becomes $(\mu^{2})$ after term extraction by Corollary \ref{thelastsheeversawofhim}.  

\medskip

Furthermore, assuming that the $\Gamma$-functional has its usual defining property $\GH^{\a}(\Gamma)$ on associates, we observe that given $\neg\paai$, the number $\Gamma(\eps, \langle\rangle)$ is nonstandard, although $\eps^{1}\in C$ and $\langle\rangle$ are standard inputs, i.e.\ $\Gamma$ is also `standard partial'  (See the proof of Theorem \ref{mukh333} for these results).   
However, $\GH_{\ns3}^{\a}$ guarantees that there is a \emph{standard-total} Gandy-Hyland functional defined on associates, and therefore $\paai$ follows, as in Theorem~\ref{mukh333}.  In short, $\GH_{\ns3}^{\a}$ guarantees that the Gandy-Hyland functional is standard-total for standard associates and standard sequences, which apparently requires $\paai$.     

\medskip

In conclusion, we have introduced the notion of `standard partiality' which allows $\P$ to accommodate the fundamental notion of `partial function(al)'.  
We have observed that $\NPC^{\a}$ as in Theorem \ref{copycatje} and $\GH_{\ns3}^{\a}$ as in Theorem~\ref{mukh333} can be viewed as principles guaranteeing the standard-totality of (functionals defined on) standard associates.  It is an interesting question whether we can fruitfully translate other theorems from computability theory regarding partial function(al)s.

\section*{Acknowledgements}
This research was sponsored by the John Templeton Foundation, the FWO Flanders, the University of Oslo, and the Alexander von Humboldt Foundation.  The author is grateful to these institutions for their support.  This work was done partially while the author was visiting the Institute for Mathematical Sciences, National University of Singapore in 2016. The visit was supported by the Institute.  
The author thanks Dag Normann and Paulo Oliva for their valuable advice.  Finally, the referees of this paper deserve thanks for their many helpful suggestions.     

\section*{References}
%


\end{document}